\documentclass[a4paper]{article}

\usepackage{
amsmath,
amsthm,
amscd,
amssymb,
}
\usepackage{stmaryrd}
\usepackage{comment}
\usepackage{tikz}
\usetikzlibrary{positioning,arrows,calc}
\usepackage{xspace}
\usepackage{yhmath}

\usepackage{graphicx}
\usepackage{boondox-cal}

\usepackage{url}
\theoremstyle{plain}
\newtheorem{lemma}{Lemma}
\newtheorem{definition}{Definition}
\newtheorem{corollary}{Corollary}
\newtheorem{proposition}{Proposition}
\newtheorem{theorem}{Theorem}

\newtheorem{question}{Question}
\newtheorem{remark}{Remark}

\newtheoremstyle{derp}% <name>
{3pt}% <Space above>
{3pt}% <Space below>
{}% <Body font>
{}% <Indent amount>
{\upshape}% <Theorem head font>
{:}% <Punctuation after theorem head>
{.5em}% <Space after theorem headi>
{}% <Theorem head spec (can be left empty, meaning `normal')>
\theoremstyle{derp}

\newcommand{\Z}{\mathbb{Z}}

\newcommand{\N}{\mathbb{N}}

\newcommand{\Grig}{\mathcal{G}}

\newcommand{\jump}{\mathcal{j}}
\newcommand{\tree}{\mathcal{t}}
\newcommand{\overgroup}{\mathcal{O}}

\newcommand{\Vee}{\mathcal{S}}
\newcommand{\Voro}{\mathcal{V}}
\newcommand{\Tree}{\mathcal{T}}

\newcommand{\act}{\blacktriangleright}
\newcommand{\str}{{\star}}

\title{SFT covers for actions of the first Grigorchuk group}

\author{Rostislav Grigorchuk \and Ville Salo}

\begin{document}
\maketitle

\begin{abstract}
We study symbolic dynamical representations of actions of the first Grigorchuk group $\Grig$, namely its action on the boundary of the infinite rooted binary tree, its representation in the topological full group of a minimal substitutive $\Z$-shift, and its representation as a minimal system of Schreier graphs. We show that the first system admits an SFT cover, and the latter two systems are conjugate to sofic subshifts on $\Grig$, but are not of finite type. % The proof is based on the simulation theorem of Barbieri for direct products. %, and generalizes to a large class of branch groups conditionally on a generalization of Barbieri's theorem for direct products.
\end{abstract}

\section{Introduction}

\subsection{Symbolic dynamics}

Classical  symbolic  systems  called    subshifts    are dynamical   systems  associated  with    actions  of   the infinite  cyclic  group $\Z$ (or  a semigroup  of  natural  numbers  $\N$). Such a  dynamical  system  %corresponding to  such  system 
is  given  by  a  pair  $(\Z, \Omega)$  where $\Omega$  is a  closed subset  of  the    space $A^{\Z}$ supplied  with  the     product  topology,  where  $A$  is a  finite  set (called the alphabet)  of cardinality  $>1$, the  group  $\Z$  acts  on  $\Omega$  by  shifts  and $\Omega$    is topologically closed and  invariant  with  respect  to  this  action. The  word  ``subshift"  addresses  the  fact  that $(\Z, \Omega)$  is a subsystem  of  the  full shift $(\Z, A^{\Z})$.

Among such  systems  a  special role is played by  subshifts  of  finite  type  (SFT), i.e.\   subshifts  determined   by  a  finite  set  of  forbidden  words (called  also strings  or   patterns)  that  are not allowed  to  appear in two-sided sequences $\omega  \in  \Omega$. A  larger  class  of  subshifts are sofic systems, which  are their factors (meaning images under a shift-commuting continuous map).  They can also be  defined as subshifts  determined  by  a set  of  forbidden  words which  constitutes  a   regular  language over  $A$.  (Regular  languages   are   the   simplest  class  of  languages  in  the  classical  Chomsky  hierarchy  of  formal  languages \cite{ChSc59}.)

A  class orthogonal (in the classical one-dimensional setting) to  the sofic  systems  is that of minimal  subshifts,  i.e.\ subshifts such that  the  orbit  of  each  point $\omega \in \Omega$ is  dense  in  $\Omega$. The  so-called  primitive  substitutions over  alphabet  $A$ (for  instance  the  famous  Thue-Morse  substitution  $0\rightarrow 01,  1\rightarrow 10$ is  an  example   of a  primitive  substitution)  give   a rich  source of  examples,  as  well  as  the  so  called  Toeplitz sequences  determining    Toeplitz systems \cite{JaKe69}. One-dimensional sofic  systems  have  plenty  of  periodic  points  whose  orbits  are  finite  and  hence are not  dense  in  $\Omega$ (assuming $\Omega$ is infinite). The   set  $\Omega$ is usually assumed  to  be  infinite  and  without  isolated  points, and in this case $\Omega$ is homeomorphic  to  a  Cantor  set.  A comprehensive  account  on  SFT  and  sofic  systems    can  be  found  in \cite{LiMa21}  while  minimal  systems  are  nicely  described  in  \cite{Ku03} for  instance.

A long  time  ago it was  observed  that  one  can  build a theory of  subshifts  associated  with an arbitrary  group  (or  semigroup)  $G$.  Usually  it  is  assumed  that  $G$ is  infinite  and  countable  or  even finitely  generated. In  this  case the ``full''  space of  interest is $\Omega=A^G$  and  the  group  $G$  naturally  acts  by  shift  of  coordinates  given  by the left  (or  right)  multiplication.  The  system  $(G, A^G)$  is  a  full  shift    and  any  closed  $G$-invariant  subset   $\Omega \subset A^G$  gives   the  subshift  system  $(G,\Omega)$. See for example \cite{Sc90,CoPa06,CeCo10,CeCo23} and references therein.

For  such  systems  one  also  can  define  a  notion  of  SFT, by  declaring  that  some  finite  set  of  patterns    is  forbidden  from   appearing in  configurations  $\omega$  from $\Omega$.  One  can  also  define  sofic  $G$-systems   as  factors of $G$-SFT  (it  seems  there  is  no  suitable alternative  definition  of them using  the terminology  of  formal  languages over the  groups).  Factoring again means  that given  two  $G$-systems  $X$  and  $Y$  there  is a  continuous $G$-equivariant surjective map  $f:X\rightarrow Y$. If  $f$  is a  homeomorphism then the systems  are said  to  be  conjugate. It  is  customary  in  topological  dynamics  to study  group  actions  up  to the  equivalence  relation of  conjugacy  similarly  like  in  group  theory  groups  are  studied  up  to  isomorphism.

The   first  considerations  of  this  type  were  given  for the  free  abelian  groups $\Z^n,  n \geq  2$   and  led  to  numerous  results  including  Berger's  result \cite{Be66} on  existence  of  SFT  for  $\Z^2$  with an   aperiodic (i.e.\ free)  action, meaning the absence  of  fixed  points for  elements  different  from the  identity.  This  was  done  through  the  construction  of an aperiodic  Wang  type  tiling  of  the plane. On these groups, contrary to $\Z$, subshifts of finite type (and thus in particular sofic shifts) can be minimal \cite{Mo89}.

The question of which groups support an SFT where the shift action is free, also known as a strongly aperiodic SFT, is widely studied, but open. A result of Jeandal \cite{Je15} shows  that algorithmic properties (complexity of the word  problem or \emph{WP}) can prevent the existence of a strongly aperiodic SFT, and another result of Cohen \cite{Co17} shows that the number of ends provides a geometric restriction. There is also a long list of individual groups and classes of groups where strongly aperiodic SFTs do exist, in particular the group $\Grig = G_\omega =\langle a, b, c, d\rangle$ constructed in \cite{Gr80} (and studied in \cite{Gr84} and many other papers, including the present one) was shown to admit a strongly aperiodic SFT in \cite{Ba19}.

\subsection{Effectively closed and totally non-free actions}

For  groups  with  nice  algorithmic  properties there is  a sense  to  study effective actions  on  zero-dimensional   compact metric  spaces,  first  of  all  on  a  Cantor  set. Roughly  speaking, the  action  of a  group $G$  with  decidable  word problem on  $X\subset A^{\N}$  is  effective if $X$ is an  effectively closed subset and for  every $g \in G$  and  $x \in X$   the  map  $x \mapsto g(x)$  is  computable  (for a precise  definition  see  Section 2).

\begin{remark}
As a word of warning, we note that in the theory of group actions, ``effective'' is sometimes used to as a synonym for ``faithful'', but in the present article, it is rather a synonym for ``computable'' (or ``lightface''), and refers to effective descriptive set theory.
\end{remark}

Arguably, a more fine-grained problem than the construction of strongly aperiodic SFTs is then to understand which effective actions of a group admit SFT covers. Namely, the existence of a free SFT on the group $G$ is trivially equivalent to the existence of any dynamical $G$-system with a free action, which admits an SFT cover (because periods are preserved under factor maps). This it is in fact a common way to produce strongly aperiodic SFTs on groups: One first finds strongly aperiodic effective subshifts on $G$, and then proves that some (or all) such subshifts are sofic. For many groups, in particular most of the groups studied in \cite{Ba19,BaSaSa21} (including the group $\Grig$), no other technique is known. This technique of \emph{simulation} is discussed in more detail in the following section.

It is well-known that subshifts are characterized among dynamical systems on zero-dimensional compact metrizable spaces by the dynamical property of expansivity \cite{He69,Me19}, so restricting to expansive zero-dimensional systems, the question is simply: which (effective) subshifts are sofic? (Our interest is mainly in the study of recursively presented groups, on which all sofic shifts are effective \cite{AuBaSa17}.)

This question is not well-understood even on abelian groups. For example, let $X \subset \{0,1\}^{\Z^3}$ be the set of configurations $x \in \{0,1\}^{\Z^3}$ such that all finite connected components of $1$-symbols (in the Cayley graph under the standard generators) have odd cardinality. It is not known whether $X$ is sofic \cite{Ho,Ba17}.

A  wide  class  of  interesting actions is given by  groups  generated  by  (asynchronous) finite automata  (viewed  as sequential  machines)  as   described  for  instance  in \cite{GrNeSu00}. This  includes  Thompson's  groups, the group $\Grig$ studied in the present paper and  many more.
At the  moment  there  is  a splash of  interest  in  the  group  actions on  a Cantor  set  inspired  by  studies  around  groups  of  Thompson-Higman  type,  groups  of  branch  type, ample (or  topological  full  groups) etc.  Usually  such  actions  are effective  and the  Cantor  set  is  realized  as a  closed  subset of  the product  space  $A^G$  where   $G$  is  infinite group  or a semigroup, or    as a  boundary  of  a  regular rooted (or  unrooted)  tree  $T$. 

%Given a  group  $G$  it  is customary  to  study what  groups  are  the  quotients  of  $G$ (or  subgroups  of  $G$.  Similarly  given a  topological  dynamical  system  $(X,G)$   it  is  natural  to  ask  what  are  the  quotients  $(Y,G)$  of  it (or  what  are subsystems of $(Y,G)$).  As an example of a successful such classification result, we mention that for substitutive $\Z$-minimal  systems there  are  only  finitely many  quotient  systems  and  all  of  them  can  be  found  by effective  procedure  \cite   Durand-Leroy22. 

%On the other hand, the  problem  of description  of  quotients  of   SFT  associated  with   groups  seems  to  be  hard  even in  the  case  of  one dimensional  subshifts  \cite LindMarkus.

%Also  given  a  particular  $G$-system $(X,G)$ the  question  if  it  could  be  covered  by  a  $G$-SFT, or  even  by the $G$-system  which  is a subsystem  of the  SFT system $(Y,H))$ with  $G$  being a  subgroup  of  $H$  could be  of  great  interest.  We  address  this  question  to  the group $\mathcal  G$  realized  by the action  on  the  boundary  $\partial{T_2}$ of  the  binary  rooted  tree   $T_2$.

A new phenomenon that appears with such group actions is that (the space of) stabilizers of points is an interesting object, and a group can have an interesting action by conjugacy on its space of subgroups. This action is of course of special interest, as it is intrinsic to the group.

An interesting related notion is the following, due to Vershik \cite{Ve11}: Consider a group action on a standard atomless measure space $X$ by a countable group $G$. Suppose further that $\mu$ is a measure on $X$ which is invariant under the action of $G$ (or at least quasi-invariant, meaning $g\mu$ and $\mu$ have the same measure zero sets for all $g \in G$). Then we say the action is called \emph{totally non-free} if the stabilizers of almost all points are distinct. We call an abstract group action \emph{totally absolutely non-free} if the stabilizers of all points are distinct.

It is known that the natural action of the group $\Grig$ on the boundary of the binary rooted tree, which we call $(\Grig,\Tree)$, is totally non-free in both senses \cite{Gr11}. This action is uniquely ergodic, meaning there is a unique probability measure that is invariant under it. The action defined in \cite{Vo12} of $\Grig$ on is marked Schreier graphs, which we call $(\Grig, \Voro)$, is also totally absolutely non-free. This is somehow implicit in \cite{Vo12}, but we give a proof in Lemma~\ref{lem:TANF}. % (its action is uniquely ergodic, so this claim does not depend on the choice of measure).
Another example is Thompson's $V$, whose natural (defining) action on Cantor space is totally absolutely non-free. It admits a quasi-invariant measure, but no invariant probability measure.

It is an interesting question which groups admit faithful totally non-free or absolutely non-free actions, and for groups that do, it is of interest to try to understand the dynamics of such actions. One concrete problem is to understand whether they can be SFTs or at least admit SFT covers. Note that this problem is in some sense of a ``complementary nature'' to the more studied problem of finding strongly aperiodic SFTs on groups -- here we want the group structure to be completely visible in the stabilizers, while in the latter we want to avoid stabilizers completely.

%As two examples of totally non-free actions we mention the defining actions of Thompson's $V$ and the (first) Grigorchuk group (by automata) on $\{0,1\}^\omega$. Note that Thompson's $V$ is only totally non-free in the second sense, as the action does not preserve any probability measure.

%The  result  of  Hochman  \cite  :  Every  effective  action  $\Z\curvearrowright X$ where  $X\subset A^{\N}$ can  be  realized  as a  topological  factor  of  the $(\Z \times \{(0,0\})$-subaction  of a $\Z^3$-SFT inspired  further  studies  in  this  direction.  

\subsection{Simulation}

We now recall a general approach for attacking problems related to soficity, called simulation. A good setting for it is that $G, H$ are groups, and $\phi : G \to H$ is an epimorphism. Then an $H$-systems can be pulled back to $G$-systems by the trivial formula $g \act x = \phi(g) \act x$.

Hochman showed in \cite{Ho09} that for any effective $\Z$-system (meaning one that can be described algorithmically, see Section~2), its pullback in $\Z^3$ admits an SFT cover (in particular, the pullback of an effective subshift is sofic). This  result  was  improved (in the expansive case)  by Aubrun and Sablik \cite{AuSa13}, and Durand, Romashchenko and Shen \cite{DuRoSh10} by  showing  that  the  two-dimensional  version  of  Hochman's  result  is  true:  every effective $\Z$-subshift is topologically conjugate to the $(\Z \times \{0\})$-subaction of a sofic $\Z^2$-subshift.

These  results inspired the  direction  of  studies  that  got  the  name \emph{simulation}, with the idea that we simulate actions in a quotient group by SFTs (up to a dynamical projection) in a covering group.  As part  of  this  terminology, a group  $G$ with  the  property  that  every effective  action of  $G$   is  a   factor  of  a  $G$-SFT  is  said  to  be  \emph{self-simulable  group}. %, the term coming from the fact that we simulate effective $G$-actions in $G$ itself.

The class of self-simulable groups was shown in \cite{BaSaSa21} to include plenty of  non-amenable groups:  direct  products $F_m \times F_n$  of  noncommutative  free  groups, the  linear  groups $\mathrm{SL}(n,\Z)$ for sufficiently large $n$,  Thompson's  group  $V$  and  many  more.  On  the  other  hand amenable  groups (in  particular $\Z^d$  and  $\Grig$)  are not self-simulable, so  for  them covering using   larger groups is necessary.  One basic reason for this is the  (topological)  entropy  theory  developed  for  countable  amenable  groups: under factorization  the  entropy  drops,  and  SFTs (and more generally subshifts) have  finite  entropy  while  any (recursively-presented) amenable  group  has an effective action  on a  compact  set  with  infinite  topological  entropy, see e.g.\ \cite{BaSaSa21}.

%For a large class of groups $G$, called \emph{self-simulable groups} [...], every effective action of $G$ directly admits an SFT cover on $G$. For example,

Since Thompson's group $V$ is self-simulable, its natural action on $\{0,1\}^\omega$ is sofic. For this, it suffices to show that the action is expansive and effective. The former is well-known, and the latter is immediate from the defining formula. This gives an example of a totally absolutely non-free sofic shift. However, this system does not preserve any non-trivial probability measure.

Since the group $\Grig$ is not self-simulable, the question of soficity of specific actions arises. %In the present paper, we study the existence of an SFT cover for its action on $\{0,1\}^\omega$, and another totally non-free action, called the Vorobets action, which represents it in the topological full group of a minimal subshift.
It turns out that simulation tools apply well to this group. Namely $\Grig$ belongs  to  the  class  of  branch  groups \cite{Gr00,BaGrSu03}, and  hence  has large product groups as a finite-index subgroups. Our main tool in this work will be the following simulation theorem for product groups, due to Sebasti\'an Barbieri \cite{Ba19}, which he used in particular to show that $\Grig$ admits a strongly aperiodic subshift of finite type.

\begin{theorem}[\cite{Ba19}]
\label{thm:Barbieri}
Let $G, H, K$ be three finitely-generated infinite groups, and $\pi : G \times H \times K \to G$ the natural projection. Then the $\pi$-pullback of any expansive effective $G$-system admits an $G \times H \times K$-SFT cover.
\end{theorem}

% We apply the simulation theorem of Barbieri from \cite{Ba19} to show that both of these systems admit an SFT cover. The Vorobets action is expansive, so we conclude that it is a sofic shift on the Grigorchuk group. We also show that it is not a subshift of finite type, unlike the natural action of Thompson's $V$.

%A finitely  generated  and  recursively  presented  group  which  admits  a  free  SFT  must  have a  decidable  word  problem.  The  latter    means  that  there  is  an  algorithm  that  given  a  word  in  the  alphabet  of  generators  and  their inversis  determines  if  the  word  represents the identity  element  or  not.

%The  question  about which    groups   with  decidable  word  problem  posses  free  SFT  is  widely  open  problem  and  can  be  addressed    for  instance  to  the  main  object  of this  article,    the  group  of  intermediate  growth  $\mathcal G$ constructed  by  the  first  author  in  \cite  Gr80.  This  group  has  decidable  word  problem  which  already  was  observed  in  ref  Gr80  and  it  can  be  described  by  the  following recursive presentations

 %  that   will appear in more  details   later  in  the  text.

\subsection{Results}

Now  we  are  ready  to  state  our  results, which deal with the group $\Grig$, in detail.

%The Grigorchuk group is the group with the presentation
%$$ \mathcal \Grig  = \langle a,b,c,d  :   1=a^2=b^2=c^2=d^2 =bcd= \sigma^n((ad)^4)=\sigma^n((adacac)^4),  n=0,1,2,\dots \rangle  $$
%where  $\sigma :a \rightarrow aca, b \rightarrow d, c\rightarrow b, d\rightarrow c$ found  by  I.\ Lysionok in \ref   .  Here  $\sigma$ is  a  substitution  application  of  iterates  of  which to  the  relators  $(ad)^4$ and  $(adacac)^4$   give  two  infinite  sequences   of  relators  involved  in  the  presentation.  Moreover,  the  presentation  \ref is   minimal   i.e.  non  of  relators  can  be  dropped  without  changing  the  group  \cite  Gr99.  The  substitution  $\sigma$ generates from $a$ the same fixed point as the primitive substitution $\tau(a) = \tau(d) = ac, \tau(b) = ad, \tau(c) = ab$, and the Vorobets action acts on the resulting minimal subshift (which embeds $\Grig$ in the topological full group).

%Here, we look at the group $\Grig$.
As mentioned, $\Grig$ acts in a natural way by  automorphisms of the binary rooted  tree $T_2$,  and  hence  on its  boundary  $\Tree = \partial T_2$  which can  be identified  with  $\{0,1\}^\omega$. The  action  on  the boundary is  by  homeomorphisms  (in  fact  by  isometries  for a suitable  ultrametric on $\Tree$).  %We call this  dynamical system System $\Tree$ (for ``tree''). 

The group $\Grig$ also acts in a natural way on a certain  closed  subset $\Vee$  of  $\{a, B, C, D\}^\Z$ defined by Vorobets in \cite{Vo10} (with slightly different choice of symbols), which is the minimal  subshift  associated  with Lysenok's  substitution from \cite{Ly85}, and which we recall in Section~\ref{sec:Grigorchuk}. Specifically, Matte Bon showed in \cite{Bo15} that $\Grig$ embeds in the topological full group of the $\Z$-shift map $\sigma_\Vee$ of $\Vee$ (another proof is given in \cite{GrLeNa17}). We also give a self-contained derivation of the embedding into the topological full group in the present paper.

We prove the following theorem.

\begin{theorem}
\label{thm:Main}
The system $(\Grig, \Vee)$ from \cite{Bo15} is topologically conjugate to a proper sofic shift on the group $\Grig$.
\end{theorem}

Recall that a sofic shift is a subshift which is a factor of a subshift of finite type, and a sofic shift is proper if it is not conjugate to a subshift of finite type. The properness, i.e.\ non-SFTness, of $(\Grig, \Vee)$, is shown in Lemma~\ref{lem:IsNotSFT}, and soficity in Lemma~\ref{lem:IsSofic}.

Vorobets studied in \cite{Vo12} a ``more efficient'' cover for $(\Grig, \Tree)$, namely an almost $1$-to-$1$ cover which we will call $(\Grig,\Voro)$. This is the same system as $(\Grig, \Vee)$ except that it is considered up to mirror-symmetry of the integer line (his system is directly a system of marked Schreier graphs), and accordingly it is a $2$-to-$1$ factor of $(\Grig, \Vee)$ \cite{GrLeNa17}. We show in Section~\ref{sec:Vorobets} that the direct analog of Theorem~\ref{thm:Main} holds for the system $(\Grig, \Voro)$. Indeed we can deduce it abstractly from the above theorem.

\begin{theorem}
\label{thm:MainVoro}
The system $(\Grig, \Voro)$ from \cite{Vo12} is topologically conjugate to a proper sofic shift on the group $\Grig$.
\end{theorem}

As $(\Grig, \Vee)$ (or $(\Grig,\Voro)$) covers $(\Grig, \Tree)$, we also get the following corollary:

\begin{corollary}
\label{cor:MainCorollary}
The system $(\Grig, \Tree)$ admits an SFT cover.
\end{corollary}

%As discussed, $(\Grig, T)$.

As discussed, our proofs are based on simulation theory. Thus, while our theorem is a nice demonstration of the power of this theory, the fact that the SFT covers come from a general construction does in practice imply that they are very complicated.

\begin{question}
Are there ``nice'' SFT covers for the systems $(\Grig, \Vee)$, $(\Grig,\Voro)$ and $(\Grig, \Tree)$? Are there ones with small fibers, e.g.\ finite-to-$1$ on a residual or full measure set?
\end{question}

We mention that in the case of the lamplighter group $\Z_2 \wr \Z$, while its natural action on $\Z_2^\Z$ is not an SFT, it has a 9-to-1 cocycle extension which is an SFT \cite[Proposition~5.14]{BaSa24}. In the case of Thompson's $V$, we do not even know whether the natural action itself is an SFT.

Theorem~\ref{thm:MainVoro} provides a natural example of a sofic shift on a finitely-generated group, which is uniquely ergodic with nonatomic measure, and where the shift-action is totally absolutely non-free (see Lemma~\ref{lem:TANF} for a proof of this property). % action (by an amenable, torsion-free finitely-generated group), both in the measure-theoretic and absolute sense (since this action is totally absolutely non-free, see Lemma~\ref{lem:TANF}), which is sofic.
We are not aware of other such examples.
It is an interesting question whether there are natural examples of totally non-free actions that are SFTs without any extension. (Of particular interest would be a uniquely ergodic action of an amenable group.)

%The following question is thus natural to ask:

\begin{question}
Is there a natural example of an SFT on a finitely-generated group $G$, which is totally (absolutely) non-free?
\end{question}

Finally, for completeness we mention that for very general reasons the system $(\Grig, \Tree)$ has the so-called pseudo-orbit tracing property, see Section~\ref{sec:TreePOTP}. Despite the fact that this property characterizes the subshifts of finite type among subshifts, as far as we can tell it cannot be used to prove Corollary~\ref{cor:MainCorollary}.
 
\begin{proposition}
\label{prop:TreePOTP}
The system $(\Grig, \Tree)$ is pseudo-orbit tracing.
\end{proposition}

\subsection{Structure of the paper}

The  paper  is  organized  as  follows. After giving some preliminaries in Section~\ref{sec:Preliminaries}, we introduce the standard actions of the group $\Grig$ in Section~\ref{sec:Grigorchuk}. This section mainly proves the well-known result that $(\Grig, \Vee)$ factors onto $(\Grig, \Tree)$ (with small fibers). We have attempted to write the section as a self-contained ``symbolic dynamical take'' on the proof. In Section~\ref{sec:BasicConstructions} we give two constructions on SFTs, namely that SFTs are in an appropriate sense closed under union and finite extensions.

In Section~\ref{sec:VSubshift} we show that $(\Grig, \Vee)$ is expansive, thus (topologically conjugate to) a subshift on the group $\Grig$. In Section~\ref{sec:VSofic} we apply the theorem of Barbieri to show that it is sofic. In Section~\ref{sec:VNotSFT} we explain why it is not a subshift of finite type. In Section~\ref{sec:Vorobets}, we show the same results for the system $(\Grig, \Voro)$. These results are obtained fully abstractly from good properties of the covering map. In Section~\ref{sec:TreePOTP} we explain why $(\Grig, \Tree)$ has the pseudo-orbit tracing property.

\subsection{Acknowledgements}

The first author is supported by the Humboldt Foundation, and expresses his gratitude to the University of Bielefeld. Also, he is supported by the Travel Support for Mathematicians grant MP-TSM-00002045 from the Simons Foundation.

\section{Preliminaries}
\label{sec:Preliminaries}

Here we specify terminology and notations used in this article.

An \emph{alphabet} is a finite discrete set $A$. We write $A^*$ for the set of all words over the alphabet, including the empty word. If $u, v$ are words, we write $u\cdot v$ or simply $uv$ for their concatenation, meaning product as elements of a free monoid. The length of a word $u \in A^*$ is the number of letters it is composed of, and is denoted $|u|$. The empty word is the unique word of length $0$.

We index a word starting from $0$, so $u = u_0u_1\cdots u_{|u|-1}$. Numbers $i$ thought of as positions in a word are called \emph{indices}. If $u \in A^*$ and $v \in A^\omega$ ($\omega$ being the first infinite ordinal), then we write $u \cdot v$ or just $uv$ for $x \in A^\omega$ where $x_i = u_i$ for $0 \leq i < |u|$ and $x_i = v_{i-|u|}$ otherwise. We say $u \in A^*$ is a \emph{subword} of $v \in A^*$ if $v = wuw'$ for some words $w, w' \in A^*$.

%We use some standard conventions from combinatorics on words (TODO: be explicit). 
% {\color{red} A minor difference from Vorobets is that my origin for bi-infinite words is to the left of coordinate $0$; Vorobets puts it on the right. His convention is more logical, but I'm used to the former.}
Intervals are by default discrete, $[i, j] = \{i, i+1, \ldots, j-1, j\}$. We have $\N = \{0,1,2,\ldots\}$.

Groups act from left. We use $G$ for a generic group, and always write the first Grigorchuk group as $\Grig$.
If $G$ is a discrete countable group, a \emph{$G$-system} is a compact metrizable zero-dimensional topological space $X$ where $G$ acts by homeomorphisms through some $\mathcal{a} : G \times X \to X$. Lowercase calligraphic letters are used for actions. We also denote the action by $\act_{\mathcal{a}}$ or just $\act$ when $\mathcal{a}$ is clear from context, i.e.\ $\mathcal{a}(g, x) = g \act_{\mathcal{a}} x = g \act x$. When $\act$ is clear from context (and especially for shift maps), we also use the shorter form $gx = g \act x$.

Two $G$-systems are \emph{(topologically) conjugate} if there is a homeomorphism between them that commutes with the actions, and the map is called a \emph{conjugacy}. Up to conjugacy, $G$-systems are just closed $G$-invariant subsets $X \subset (\{0,1\}^\N)^G$ with action given by the shift formula $g \act x_h = x_{hg}$. Factor maps (resp.\ embeddings) are $G$-commuting surjective (resp.\ injective) continuous maps from one $G$-set onto (resp.\ into) another. A preimage in a factor map is called an \emph{extension} or a \emph{cover}.

An \emph{effective} $G$-system is a $G$-system such that $X \subset \{0,1\}^\N$ is effectively closed, and the relation $R = \{(x, g \act x) \;|\; x \in \{0,1\}^\N\}$ is effectively closed in $(\{0,1\}^\N)^2$ for all $g \in G$. Recall that a subset of $\{0,1\}^\N$ is effectively closed if we can recursively enumerate a sequence of words such that the union of the corresponding cylinders is the complement of the subset; and for a relation $R$, effectively closed means that we can enumerate a sequence of pairs of words $(u_i, v_i) \in (\{0,1\}^*)^2$ such that $\bigcup_i [u_i] \times [v_i]$ gives the complement of $R$. The notion of an effective $G$-system is equivalent to saying that $X \subset \{0,1\}^\N$ is effectively closed and for every $g \in G$, the map $(x, n) \mapsto (g \act x)_n$ can be computed by a Turing machine (depending on $g$), with $x$ given as an oracle, and $n$ as input.

An \emph{effective $G$-subshift} is simply a subshift that happens to be effective as a $G$-system. For $G$ a finitely-generated group with decidable word problem, it is very transparent what this means. The group $G$ is in computable bijection with $\N$ (by representing $g \in G$ first as the lexicographically minimal word over a generating set, and then bijecting such words with $\N$), and the group operations are computable functions on $\N$. In particular, we can see configurations in $A^G$ as elements of $A^\N$, and the shift maps are easily seen to be computable. Finally, one may code the system to the binary alphabet $\{0,1\}$ by replacing letters by binary words of a fixed length.

%When $G$ is finitely-generated and has decidable word problem, there is a computable bijection between $G$ and $\N$, in the sense that we can give a normal form for each element of $G$ as a word over a finite generating set, and then we can order the representatives lexicographically. Then an effective $G$-system is conjugate to an effectively closed set $X \subset (\{0,1\}^\N)^G$ where $G$ again acts by the shift formula. %, i.e.\ $((g \act x)_h)_n = (x_{hg})_n$.

A \emph{pattern} is a function $P : D \to A$ where $D \subset G$ is finite. We say a pattern $P$ \emph{appears in $x \in A^G$ at $g$} if $g \act x \in [P]$, where $[P] = \{y \in A^G \;|\; y|_D = P\}$ is the \emph{cylinder} defined by $P$. We simply say it appears if it appears at $g$ for some $g$. We say $P$ appears in a set of configurations $X \subset A^G$ if it appears in some $x \in X$.

A \emph{$G$-subshift} over alphabet $A$ is a topologically closed and $G$-invariant subset of $X \subset A^G$, where $G$ acts on $A^G$ by $gx_h = x_{hg}$. Equivalently, $X$ is defined by forbidden patterns, in the sense that for some patterns $P_1, P_2, \ldots$, we have $X = \{x \in A^G \;|\; \forall i: P_i \mbox{ does not appear in } x\}$. A \emph{$G$-SFT} is a $G$-subshift defined by finitely many forbidden patterns. A \emph{sofic $G$-shift} is a subshift that is a factor of an SFT. A \emph{proper sofic shift} is a sofic shift that is not conjugate to an SFT. Here, an \emph{effective subshift} is one that is topologically conjugate to an effective system. When $G$ is finitely-generated with decidable word problem, this is equivalent to the subshift being definable by a recursively enumerable list of forbidden patterns $P_1, P_2, \ldots$.

Up to topological conjugacy, being a subshift is equivalent to being \emph{expansive}, i.e.\ for some $\epsilon > 0$, $\forall x \neq y: \exists g \in G: d(gx, gy) > \epsilon$. It is well-known that the SFT property is characterized among subshifts by the so-called pseudo-orbit tracing property \cite{ChLe18}, which we recall next.

%\begin{definition}
If $\delta > 0$ and $F \Subset G$ is finite, then a \emph{$(\delta, F)$-pseudo-orbit} for the system $(G, X)$ is a map $z : G \to X$ satisfying $d(sz(g), z(sg)) < \delta$ for all $g \in G, s \in F$. A point $x \in X$ \emph{$\epsilon$-traces} (or \emph{$\epsilon$-shadows}) a pseudo-orbit $z$ if $d(gx, z(g)) < \epsilon$ for all $g \in G$. We say $(G, X)$ has the \emph{pseudo-orbit tracing property} (a.k.a.\ the \emph{shadowing property}) if for all $\epsilon > 0$, there exists $\delta > 0$ and $F \Subset G$ such that every $(\delta, F)$-pseudo-orbit is $\epsilon$-traced by some $x \in X$.
%\end{definition}

The following is classical \cite{Wa77,ChLe18}.

\begin{lemma}
For any group $G$, a $G$-subshift has the pseudo-orbit tracing property if and only if it an SFT.
\end{lemma}

%Being SFT is equivalent to being \emph{shadowing}, i.e.\ $\forall \epsilon > 0: \exists \delta > 0$ such that all $\delta$-pseudo-orbits are $\epsilon$-shadowed by orbits, where a \emph{$\delta$-pseudo-orbit} is a map $\pi : G \to X$ satisfying $d(h\pi(g), \pi(hg)) < \delta$, and $x$ \emph{$\epsilon$-shadows} $\pi$ if $d(gx, \pi(g))< \epsilon$ for all $g \in G$.

Being sofic has no known simple characterization on most groups, the exceptions being that
\begin{itemize}
\item it is understood well on $\Z$ \cite{LiMa21} where sofic shifts are given by labelings of bi-infinite paths in a finite directed graph, and
\item on self-simulable groups \cite{BaSaSa21} with decidable word problem, a subshift $X \subset A^G$ is sofic if and only if the shift action of $G$ is effective (in the sense of being an effective $G$-subshift, as discussed above).
\end{itemize}

A subshift $X \subset A^G$ is \emph{minimal} if it is nonempty and any subshift $Y \subset X$ is either $X$ or empty. This is equivalent to the property that the orbit of every point is dense, i.e.\ every nonempty open set is visited by every orbit. Then in fact for all nonempty open sets $U$ (in particular for all nonempty cylinder sets) there exists a finite set $S \Subset G$ such that for all $x \in X$, we have $gx \in U$ for some $g \in S$.

Here we study which systems admit an SFT cover. Such systems are sometimes called ``finite type'' in the literature, but this terminology can be confusing, since a subshift which is of finite type need not be a subshift of finite type. It would also make sense to call these systems ``sofic'', but there is a possibility of confusion here as well, since ``sofic system'' often refers to a ``sofic subshift'' in the literature. Thus, if $X$ has an SFT cover, we will simply call it \emph{SFT covered}.

On $A^\Z$ we use the metric $d$ where for $x \neq y$ we define $d(x, y) = \inf \{2^{-n} \;|\; n \in \N, x_{[-n, n]} = y_{[-n, n]}\}$.

\begin{definition}
Let $(G, X)$ be a group acting on a compact topological space $X$. The \emph{topological full group} $\llbracket (G, X) \rrbracket$ consists of homeomorphisms $h$ such that there exists a partition of the space $X$ into clopen sets $A_1, \ldots, A_k$ and group elements $g_1, \ldots, g_k$ such that $h|_{A_i} = g_i|_{A_i}$ for all $i$.
\end{definition}

The data above can be summarized into a single continuous function $\gamma : X \to G$:
we consider our groups $G$ to be discrete, so the image, being compact, is finite.
Then $A_g = \gamma^{-1}(\{g\})$ forms the clopen partition as $g$ ranges over $G$, and
we define $h|_{A_g} = g|_{A_g}$. Such $\gamma$ is called a \emph{cocycle}.

We also use the notation of Vorobets \cite{Vo20} in the case of a $\Z$-subshift $(\sigma, X)$, where $\sigma$ now denotes the left shift $\sigma(x)_i = x_{i+1}$ (which corresponds to the the action of $-1 \in \Z$ with our general formula).
If $U$ is a clopen set such that $U \cap \sigma(U) = \emptyset$, we define a homeomorphism
$\delta_U : X \to X$ by
\[ \delta_U(x) = \left\{\begin{array}{ll}
\sigma(x) & \mbox{if } x \in U \\
\sigma^{-1}(x) & \mbox{if } x \in \sigma(U) \\
x & \mbox{otherwise.}
\end{array}\right. \]

If a $\Z$-subshift $(\sigma, X)$, $X \subset A^\Z$ for finite alphabet $A$, is clear from context, and $u \in A^*$, $i \in \Z$, then we also use a notation for positioned cylinders $[u]_i = \{x \in X \;|\; x_{[i,i+|u|-1]} = u\}$. Such a set is a cylinder in the previous more general sense, so it is clopen. The \emph{language} of a $\Z$-subshift $X$ is the set of words $w \in A^*$ such that $[w]_0$ is nonempty. A $\Z$-subshift $X$ is minimal if and only if for all words $u$ in the language of $X$, there exists $n$ such that all words of length $n$ in the language of $X$ contain $u$ as a subword.

\section{The group $\Grig$ and its actions}
\label{sec:Grigorchuk}

We now give a self-contained construction of the systems $(\Grig, \Tree), (\Grig, \Vee)$ and $(\Grig, \Voro)$. The first is the ``defining'' action on a tree boundary, the second is the topological full group action first defined by Matte Bon in \cite{Bo15}, and the third is Vorobets' action on Schreier graphs from \cite{Vo12} (which we present slightly differently). These sit in a chain of factoring relations (as explained in \cite{GrLeNa17})
\[
\tikz{
\node (A) at (0, 0) {$(\Grig,\Vee)$};
\node (B) at (2, 0) {$(\Grig,\Voro)$};
\node (C) at (6, 0) {$(\Grig,\Tree)$};
\draw (A) edge[->] node[above] {2:1\phantom{,}} (B);
\draw (B) edge[->] node[above] {3:1, almost 1:1} (C);
%\Vee \overset{\mathrm{2:1}}{\longrightarrow} \Voro \overset{\overset{3:1}{\mathrm{\scriptscriptstyle{almost \;}} 1:1}}{\longrightarrow} \Tree
}
\]

\subsection{Action on (the boundary of) a tree}

Write $\Tree = \{0,1\}^\omega$. We interpret $\Tree$ as the boundary $\partial T_2$ of a binary rooted tree $T_2$, and with the product topology $\Tree$ is homeomorphic to Cantor space.

The group $\Grig$ is generated by homeomorphisms $a,b,c,d : \Tree \to \Tree$ defined as follows: The homeomorphism $a$ acts by $a \act \alpha x = (1 - \alpha) x$, where $\alpha \in \{0,1\}$, $x \in T$. The homeomorphism $g \in \{b, c, d\}$ acts on a dense set of points by
\[ g \act 1^n 0 \alpha x = \left\{\begin{array}{ll}
1^n 0 (1 -\alpha) x & \mbox{ if } n \not\equiv v_g \bmod 3, \\
1^n 0 \alpha x & \mbox{ otherwise.} \\
\end{array}\right. \]
where $v_b = 2, v_c = 1, v_d = 0$, and $g \act 1^\N = 1^\N$ for $g \in \{b,c,d\}$. This extends naturally and uniquely to a continuous action on $\Tree$.

\begin{definition}
We call this defining action the \emph{tree action}, and (when not clear from context) denote it by $\act_\tree$.
\end{definition}

%\begin{definition}
%$(\Grig, \Vee)$ covers $(\Grig, \partial T_2)$
%We call $(\Grig, T)$ System $T$, where $\Grig$ acts by $\act_{\tree}$.
%\end{definition}

The system $(\act_{\tree}, \Grig, \Tree)$ is of course a $\Grig$-system, i.e.\ a compact metrizable zero-dimensional space with a continuous action of $\Grig$.  It is useful to also define an action on $\{0,1\}^m$ by the analogous formula (i.e.\ the quotient action when we project to the first $m$ coordinates), and we use the same name for it. Note that elements of $\{0,1\}^m$ correspond to vertices of $T_2$ at height $m$, and the action is by tree automorphisms.

\begin{comment}
By similar formulas, we define another group called the \emph{overgroup}, denoted $\overgroup$. It is generated by $a, B, C, D$ where $a$ has the same action, and
\[ g \act 1^n 0 \alpha x = \left\{\begin{array}{ll}
1^n 0 (1 -\alpha) x & \mbox{ if } n \equiv v_g \bmod 3, \\
1^n 0 \alpha x & \mbox{ otherwise.} \\
\end{array}\right. \]
where $v_b = 2, v_c = 1, v_d = 0$, and of course $g \act 1^\N = 1^\N$ for $g \in \{b,c,d\}$. Clearly $a \mapsto a, b \mapsto CD, c \mapsto BD, d \mapsto BC$ embeds the group $\Grig$ into the overgroup (with a compatible action).
\end{comment}

\subsection{Action on a $\Z$-subshift}

We next define another action which is more directly related to symbolic dynamics. Recall that the group $\Grig$ is not finitely-presented, but can be given a natural infinite (EDT0L \cite{CiElFe18}) presentation
\[ \mathcal \Grig  = \langle a,b,c,d  :   a^2, b^2, c^2, d^2, bcd,  \kappa^n((ad)^4), \kappa^n((adacac)^4),  n=0,1,2,\dots \rangle \]
where  $\kappa :a \rightarrow aca, b \rightarrow d, c\rightarrow b, d\rightarrow c$ is a substitution (= free monoid endomorphism), the application  of  iterates  of  which to  the  relators  $(ad)^4$ and  $(adacac)^4$   gives  two  infinite  sequences   of  relators  involved  in  the  presentation. This presentation was found  by  I.\ Lysenok in \cite{Ly85}, and it was shown in \cite{Gr99} to be minimal, i.e.\ none  of the relators  can  be  dropped  without  changing  the  group. The  substitution $\kappa$ generates from $a$ the same fixed point as the primitive substitution $\tau(a) = \tau(d) = ac, \tau(b) = ad, \tau(c) = ab$. This is a minimal $\Z$-subshift. % (which embeds $\Grig$ in the topological full group).

We next represent $\Grig$ by bijections on certain ``starred words'', %. This is a rephrasing of the \cite{Vo10,Vo12,Ma14,Vo20}, and shows the claim from
with the goal of proving that $\Grig$ embeds in the topological full group of the substitutive subshift $\Vee$ given by $\kappa$ (or $\tau$), first observed in \cite{Bo15}. % we rederive the usual topological full group embedding.
(Our subshift uses upper case letters and a different permutation of $b, c, d$, than \cite{Bo15} does.)

\begin{remark}
We give the construction in detail, mostly to be self-contained, but also since we need some details in the proofs of Lemma~\ref{lem:IsNotSFT} and Lemma~\ref{lem:PointwiseInV}. Of these, only Lemma~\ref{lem:PointwiseInV} is needed for what we consider the main part of the proof, namely that $(\Grig, \Vee)$ is sofic.
\end{remark}

An $a$-Alternating Word, or \emph{AW} for short, is a word $w$ in $\{a,B,C,D\}^*$ (i.e.\ a finite-length, possibly empty, word over this cardinality-$4$ alphabet) such that in every subword $\alpha \beta$ of $w$ with $\alpha,\beta \in \{a,B,C,D\}$ exactly one of $\alpha, \beta$ is equal to $a$, i.e.\
\[ (\alpha = a \vee \beta = a) \wedge (\alpha \neq a \vee \beta \neq a). \]
Equivalently $w$ is an AW if every $a$ (non-$a$) symbol is followed by a non-$a$ (resp.\ $a$) symbol (unless we are at the end of the word), and vice versa. For example, $aBa$, $a$ and the empty word are AWs, while $aBBa$ is not.

A Starred $a$-Alternating Word or \emph{SAW} is a word $w$ in $\{a,B,C,D,\str\}^n$ for arbitrary $n$, where exactly one $\str$ appears, and after we remove it we obtain an AW. For example $\str, a\str B$ are SAWs, but $a \str a$ is not. Let $W$ be the set of all SAWs. \emph{Starrings} of an AW are the SAWs obtained by inserting the star-symbol anywhere in the AW, and the resulting index of the star is the \emph{star position}. For an AW of length $k$, the possible star positions are $[0, k]$.

The star $\str$ can be thought of as a movable ``origin'' (which later in the topological full group context will literally translate to an origin). Other common symbols used for such a purpose are the decimal dot and the symbol $|$. % we use this symbol instead of the perhaps more standard decimal dot for notational clarity. The symbol $|$ is also often used for such a purpose.

Geometrically, a word like $w = a D a \str C a D a$ may be pictured as
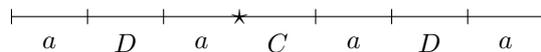
\begin{center}
\tikz{
\draw (0,0) -- (7,0);
\def\ee{0.1}
\foreach \i in {0,1,2,4,5,6,7} {
\draw (\i,-\ee) -- (\i,\ee);
}
\node () at (0.5, -0.35) {$a$};
\node () at (1.5, -0.35) {$D$};
\node () at (2.5, -0.35) {$a$};
\node () at (3.5, -0.35) {$C$};
\node () at (4.5, -0.35) {$a$};
\node () at (5.5, -0.35) {$D$};
\node () at (6.5, -0.35) {$a$};
\node () at (3, 0) {\large $\str$};
}
\end{center}
i.e.\ replacing letters as labeled intervals of length $1$, where the $\str$ is between two of the intervals (or at one of the ends). %It is also useful to think of the $\str$ as marking the origin of the word: this view makes the connection to both Schreier graphs and the topological full group transparent.

Consider the free product $H = \Z_2 * \Z_2^2$ where the $\Z_2$ generator is again $a$ and $b,c,d$ are the nontrivial elements of $\Z_2^2$. As per the definition of a free product, specifying an action of $H$ is the same as specifying an action of $\Z_2$, and an action of $\Z_2^2$, which can be entirely independent. Recall that every element $h \in H$ can be put in \emph{freely reduced form}, meaning in the form $a^x \alpha_1 a \alpha_2 \cdots \alpha_k a^y$ where $\alpha_i \in \{b, c, d\}$ and $x, y \in \{0,1\}$, by simply performing multiplications and reductions greedily. These freely reduced forms are in one to one correspondence to the AWs, but play a somewhat different role.

Then $H$ acts on the set of SAWs $W$ by what we call the \emph{jump action}, denoted $\act_\jump$, which we now define. The idea is that the generator $a$ will have the unique $\str$ jump over any $a$ adjacent to it if one is present in the sense that if the subword $a \str$ appears, it is rewritten into $\str a$ and vice versa. We write ``if present'' since of course it is possible that there is no $a$-symbol adjacent to $\str$. Similarly, $b$ has it jump over any adjacent $C, D$ if present (i.e.\ it similarly jumps over either of these symbols); $c$ jumps over $B,D$; $d$ jumps over $B, C$.

In formulas, the definition is as follows, where $\alpha, \beta, \in \{a, B, C, D\}$ and $u, v \in \{a, B, C, D\}^*$:
\begin{align*}
\alpha \act_\jump u \str \beta v &= u \beta \str v \mbox{\; if $\beta \in S_\alpha$}, \\
\alpha \act_\jump u \beta \str v &= u \str \beta v \mbox{\; if $\beta \in S_\alpha$ } \\ % and the previous formula cannot be applied}, \\
\alpha \act_\jump u \str v \;\,\, &= u \str v \mbox{\;\;\,\, if neither of the above formulas can be applied.}
\end{align*}
Here, $S_a = \{a\}, S_b = \{C, D\}, S_c = \{B, D\}, S_d = \{B, C\}$. Note that the first and second cases cannot overlap for SAWs. %, if the second formula applies then the third formula does not.

It is easy to verify that $a$ acts by an involution and $\langle b, c, d \rangle$ satisfy the identities of $\Z_2^2$, so this is indeed a well-defined action of $H$. Of course, this is not an action of $\Grig$. For example, $ab$ is (like all other elements) of finite order in $\Grig$, but $(ab)^n \act_\jump (aD)^n \str = \str (aD)^n$ for all $n$ in the action of $H$.

Define a list of AWs inductively by $w_1 = a$, and in general $w_{n+1} = w_n \alpha w_n$ where
\begin{equation}
\label{eq:AlphaChoice}
\alpha = \left\{ \begin{array}{ll}
B & \mbox{ if } n \equiv 0 \bmod 3 \\
D & \mbox{ if } n \equiv 1 \bmod 3 \\
C & \mbox{ if } n \equiv 2 \bmod 3 \\
\end{array}\right.
\end{equation}
Checking that these are AWs is an immediate induction. Note that $|w_n| = 2^n-1$, thus the possible star positions of starrings of $w_n$ are $[0, 2^n-1]$.

For explicitness, we enumerate a few initial words from this list:
\[ w_1 = a, w_2 = aDa, w_3 = aDaCaDa, w_4 = aDaCaDaBaDaCaDa, ... \]
We illustrate the $H$-action of $dadaba$ on a starring of $w_3$;
\begin{align*}
dadaba \act_\jump aDa\str CaDa &= dadab \act_\jump aD\str aCaDa \\
&= dada \act_\jump a\str DaCaDa \\
&= dad \act_\jump \str aDaCaDa \\
& = da \act_\jump \str aDaCaDa \\
& = d \act_\jump a\str DaCaDa \\
& = a \str DaCaDa
\end{align*}
Geometrically, the $\str$ jumps around as follows:
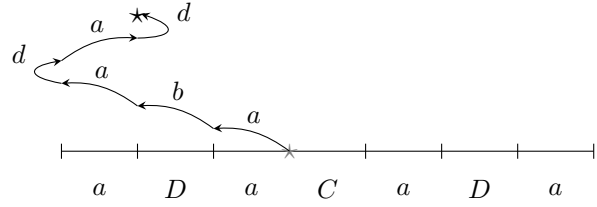
\begin{center}
\tikz{
\draw (0,0) -- (7,0);
\def\ee{0.1}
\foreach \i in {0,1,2,4,5,6,7} {
\draw (\i,-\ee) -- (\i,\ee);
}
\node () at (0.5, -0.5) {$a$};
\node () at (1.5, -0.5) {$D$};
\node () at (2.5, -0.5) {$a$};
\node () at (3.5, -0.5) {$C$};
\node () at (4.5, -0.5) {$a$};
\node () at (5.5, -0.5) {$D$};
\node () at (6.5, -0.5) {$a$};
\def\dd{0.3}
\node () at (3, 0) {\color{gray} \large $\str$};
\draw (3, 0) edge [bend right = 20, -stealth] node [above] {$a$} (2, \dd);
\draw (2, \dd) edge [bend right = 20, -stealth] node [above] {$b$} (1, \dd*2);
\draw (1, \dd*2) edge [bend right = 20, -stealth] node [above] {$a$} (0, \dd*3);
\draw (0, \dd*3) edge [bend left = 80, looseness = 4, -stealth] node [above left] {$d$} (0, \dd*4);
\draw (0, \dd*4) edge [bend left = 20, -stealth] node [above] {$a$} (1, \dd*5);
\draw (1, \dd*5) edge [bend right = 80, looseness = 4, -stealth] node [above right] {$d$} (1.05, \dd*6+0.02);
\node () at (1, \dd*6) {\large $\str$};
}
\end{center}

\begin{lemma}
\label{lem:ActionOnStarrings}
For each $n$, the group $\Grig$ admits a well-defined action on the set of starrings of $w_n$ by $\act_\jump$. Furthermore, for a fixed $n$, this action is conjugate to its tree action on $\{0,1\}^n$.
\end{lemma}

\begin{proof}
We show a bit more than claimed in the lemma: the words $w_n$ are palindromes, they begin and end with $a$, and the conjugacy $\phi_n : [0, 2^n-1] \to \{0,1\}^n$, where $[0, 2^n-1]$ are interpreted as star positions in starrings of $w_n$, can be chosen so that $\phi_n(0) = 1^n$ and $\phi_n(2^n-1) = 1^{n-1} 0$. Clearly this is true for $n = 1$: The tree action $\act_\tree$ of level $1$ (i.e.\ on $\{0,1\}$) is clearly isomorphic to the jump action on the cycle of two SAWs $W_1 = (\str a \; a\str)$; and indeed $\str a$ corresponds to $1$ and $a \str$ corresponds to $0$.

Suppose we have shown all these claims for level $n$. We immediately note that $w_{n+1} = w_n \alpha w_n$ is a palindrome, and begins and ends in the symbol $a$.

We define $\phi_{n+1}$ inductively based on $\phi_n$. The idea is to observe that during the tree action orbit of $1^n 1$ (following the unique trajectory), in the first $n$ bits we first follow the level-$n$ orbit of $1^n$ forward until the end (i.e.\ $1^{n-1} 0$), then we flip the ``new bit'' at depth $n$ if and only if the suitable generator is used (and this bit cannot be affected at any other point), then in the first $n$ bits we retrace the orbit back from $1^{n-1} 0$ to $1^n$. A moment's reflection shows that, since $w_n$ is a palindrome, this conjugates the actions.\footnote{Readers familiar with the Gray code may find it useful to observe that $\phi_{n+1}$ is constructed from $\phi_n$ according to the inductive definition of the Gray code, up to the bit flip $0 \leftrightarrow 1$, and indeed $(\phi_n(j))_{j = 0}^{2^n-1}$ is the standard Gray code up to this bit flip.}

%We illustrate the definition in Figure~\ref{fig:InductiveDefinition}.

%\begin{figure}
%\begin{center}
%\includegraphics[scale=0.3]{GrayConstruction}
%\end{center}
%\caption{An illustration of the construction of $\phi_n$ inductively in $n$. {\color{red} I will draw a picture later, for now just an unreadable sketch.}}
%\label{fig:InductiveDefinition}
%\end{figure}

In the remainder of the proof, we explain the same in formulas. %While this looks technical, we emphasize that these formulas flow entirely deterministically from the description. %We temporarily introduce notation for starrings of $w_n$. % with the set $[0, 2^n-1]$. I.e.\
Denote (temporarily) the $j$th starring of $w_n$ by $s_n(j)$, i.e.\ if $w_n = u \cdot v$ where $|u| = j$, then $s_n(j) = u \str v$. For $g \in H$ and $j \in [0, 2^n-1]$, it will be convenient to use the notation $g \act_n j = s_n^{-1}(g \act_{\jump} s_n(j))$. I.e.\ $j$ can represent the $j$th starring of any $w_n$ with $2^n > j$, and the subscript of $\act$ tells us which $w_n$ we use when acting. %It might seem natural to write $n$ as a subscript on the number, so that $0_n$ would represent $\str w_n$ and $(2^n-1)_0$ would represent $w_n \str$, but for notational purposes it turns out to be more convenient to write the subscript on the action. Thus, we define $g \act_n j = $.

 % this has the ``downside'' that a number $j$ on its own does not determine which word's starring we are discussing, but when we talk about multiple levels at once this actually turns out to be an upside, as it avoids explicit nested type conversion. We really only need to know which level we are discussing to know how to act, so a subscript on the group element is sufficient.

The description two paragraphs above suggests that if $\phi_n(j) = u$, then we should let $\phi_{n+1}(j) = u1$ and $\phi_{n+1}(F(j)) = u0$ where $F(j) = F_{n+1}(j) = 2^{n+1} - 1 - j$. One should think of $F$ as flipping starring positions around the middle of the central $\alpha$-symbol in $w_{n+1} = w_n \alpha w_n$. Equivalently, for $j \in [0, 2^n-1]$ we have $\phi_{n+1}(j) = \phi_n(j) 1$, and for $j \in [2^n, 2^{n+1}-1]$ we have $\phi_{n+1}(j) = \phi_n(F(j)) 0$. The formulas clearly define all the values of $\phi_{n+1}$. To see that this conjugates the actions, let $g$ be any of the generators and consider one of the starrings $j \in [0, 2^{n+1} - 1]$.

First, suppose $j, g \act_{n+1} j \in [0, 2^n-1]$. Then we have
\begin{align*}
\phi_{n+1}(g \act_{n+1} j) &= \phi_{n+1}(g \act_n j) \\
&= \phi_n(g \act_n j) 1 \\
&= g \act_{\tree} \phi_n(j) 1 \\
&= g \act_{\tree} \phi_{n+1}(j)
\end{align*}
Here, the first equality holds because $w_{n+1}$ begins with $w_n$. The second holds because $g \act_n j = g \act_{n+1} j$ holds, which in turn is true because
$j, g \act_{n+1} j \in [0, 2^n-1]$, and by the definition of $\phi_{n+1}$. %, and the $\act_{n+1}$-action covers the $\act_n$ action.
The third holds because $\phi_n$ is a conjugacy so $\phi_n(g \act_n j) = g \act_{\tree} \phi_n(j)$, and the last bit is not affected in the application $g \act_{\tree} \phi_n(j_n) 1$ because it can only be affected if $\phi_n(j) = 1^{n-1} 0$, and in this case only if $\alpha \in S_g$. But then we would have $g \act_{n+1} j = 2^n$, contradicting the assumption $g \act_{n+1} j \in [0, 2^n-1]$. The fourth equality is by definition.

Next, suppose $j, g \act_{n+1} j \in [2^n, 2^{n+1} - 1]$. %It helps to write $j_{n+1} = 2^{n+1} - 1 - k$ and.
Then we have the exact same calculation, though it looks more complicated. We first calculate $g \act_{n+1} j$ in terms of the action on starrings of $w_n$. %Observe that the $(n+1)$st level starring $j$ corresponds to the $n$th level starring $F(j)$ in the sense that this is where we copy the $\phi_{n+1}$-value-, and this is also where $w_{n+1}$ can be seen as getting its $j$th symbol in the recursive definition (since $w_n$) is a palindrome. So
If $g \act_n F(j) = F(j) + k$ then $g \act_{n+1} j = j - k$, because $w_n$ is a palindrome and because by assumption $j - k \in [2^n, 2^{n+1} - 1]$. This gives
\begin{align*}
g \act_{n+1} j &= j - k \\
&= j - (g \act_n F(j) - F(j)) \\
&= j + F(j) - g \act_n F(j) \\
&= 2^{n+1} - 1 - g \act_n F(j) \\
&= F(g \act_n F(j))
\end{align*}

Thus,
\begin{align*}
\phi_{n+1}(g \act_{n+1} j) &= \phi_{n+1}(F(g \act_n F(j))) \\
&= \phi_n(g \act_n F(j)) 0 \\
&= (g \act_{\tree} \phi_n(F(j))) 0 \\
&= g \act_{\tree} \phi_{n+1}(j) \\
\end{align*}
Here the first equality was calculated above. The second is because $F$ is an involution and $F(g \act_n F(j)) \in [2^n, 2^{n+1} - 1]$ (which in turn holds because $g \act_n F(j) \in [0, 2^n-1]$, which is clear). The third equality is what we showed in the previous case. The fourth is by the definition of $\phi_{n+1}$ for inputs $\geq 2^n-1$, and the fact that $j, g \act_{n+1} j \in [2^n, 2^{n+1} - 1]$ implies that the tree action does not flip the final bit of $\phi_{n+1}(j)$.

Finally assume $j \in [0, 2^n-1]$ and $g \act_{n+1} j \notin [0, 2^n-1]$ (the case of $j \notin [0, 2^n-1]$ and $g \act_{n+1} j \in [0, 2^n-1]$ being symmetric). Then of course $j = 2^n-1$ and $g \act_{n+1} j = 2^n$, and if $w_{n+1} = w_n \alpha w_n$, then $g \in \{b, c, d\}$ and $\alpha \in S_g$. Direct calculations give
\[ \phi_{n+1}(g \act_{n+1} j) = \phi_n(F(2^n)) 0 = \phi_n(2^n - 1) 0 = 1^{n-1} 0 0 \]
and
\[ g \act_{\tree} \phi_{n+1}(j) = g \act_{\tree} \phi_n(j) 1 = g \act_{\tree} 1^{n-1} 0 1 = 1^{n-1} 0 0 \]
because the fact that $\alpha \in S_g$ implies the last equality here, by the definition of the tree action.

This concludes the proof.
\end{proof}

We can now easily derive the usual topological full group action.

\begin{definition}
\label{def:Vee}
Let $\Vee \subset \{a, B, C, D\}^\Z$ be the smallest $\Z$-subshift whose language contains the words $w_n$.
\end{definition}

\begin{lemma}
\label{lem:Language}
The subshift $\Vee$ is well-defined, and the following are equivalent for a word $u$ with $u \leq |w_n|$:
\begin{itemize}
\item $u$ is in the language of $\Vee$,
\item $u$ is a subword of some $w_m$,
\item $u$ is a subword of $w_{n+3}$,
\item $u$ is a subword of $w_n \alpha w_n$ for some $\alpha \in \{B, C, D\}$.
\end{itemize}
\end{lemma}

\begin{proof}
We claim that the set $U$ of all subwords of the words $w_n$ is extendable, meaning one can always find a word of the form $uw_nv$ in $U$ with $|u|, |v| > 0$. The properties of extendability and closure under taking subwords are $U$ known to characterize languages of subshift, so this shows that $\Vee$ is well-defined and that $U$ is its language (because any subshift containing the words $w_n$ would certainly also have all words of $U$ in its language). In other words it shows that equality of the first two items. For extendability of $U$, it suffices to show that $w_n$ itself is extendable within $U$, and this is clear because $w_{n+2} = w_n \alpha w_n \beta w_n \alpha w_n$ where $\alpha, \beta \in \{B, C, D\}$.

It is clear that the third item only describes words in $U$. To see that the fourth item also only describes words in $U$, observe that
\begin{equation}
\label{eq:third} w_{n+3} = w_n \alpha w_n \beta w_n \alpha w_n \gamma w_n \alpha w_n \beta w_n \alpha w_n \end{equation}
where $\alpha, \beta, \gamma$ are distinct.

Suppose now that $u \in U$ and $u$ with $u \leq |w_n|$. The claim in the fourth item clearly holds, because we see from the inductive definition that all words $w_k$ fit the regular expression $(w_n (B + C + D))^* w_n$, and all subwords of such words, which have length at most $|w_n|$, fit inside subword of the form $w_n \alpha w_n$ for some $\alpha \in \{B, C, D\}$. Finally, \eqref{eq:third} shows that $w_{n+3}$ already contains all these words.
\end{proof}

The following is proved in \cite{Vo10}.

\begin{lemma}
\label{lem:MinimalNoPeriodic}
The subshift $\Vee$ (under the shift action of $\Z$) is minimal and has no periodic points.
\end{lemma}

\begin{proof}
If $u$ appears in $\Vee$, then it appears in some $w_n$ by definition. As observed in the previous proof, all words in the language of $\Vee$ are subwords of words in $(w_n \{B, C, D\})^* w_n$, thus any word of length $2|w_n| + 1$ contains the word $u$. For aperiodicity, if a minimal subshift has a $p$-periodic point, then all its points are easily seen to be $p$-periodic. But obviously if $|w_n| \geq p$, at most one of the words $w_n \alpha w_n$ can be $p$-periodic for $\alpha \in \{B, C, D\}$, and all of these words appear in the language.
\end{proof}

Another way to see the previous lemma is to observe that $\Vee$ can also be generated by a primitive substitution \cite{Vo20}, as it is well-known that the subshift generated by a primitive substitution is minimal and has no periodic points. In \cite{Vo10} it is in turn deduced from the fact that this subshift is the shift orbit closure of a Toeplitz sequence. %Note that while $V$ is a $\Z$-subshift in the literal sense of bi-infinite words, and

\begin{lemma}
The group $\Grig$ embeds in the topological full group of $(\sigma, \Vee)$, with the action induced by the jump action $\act_\jump$ on the set of SAWs $W$.
\end{lemma}

\begin{proof}
%It is immediate from the definition of the words (and the fact they are palindromes) that $V$ is well-defined, and its language contains only subwords of the $w_n$.
To get the cocycles of the elements of $\Grig$, we simply mimic the jump action, and shift the origin as if it were the star. More precisely, the generator $\alpha \in \{a, b, c, d\} \subset H$ is precisely $\delta_{U_\alpha}$, where
\begin{comment}
 the cocycle $\gamma_\alpha$ of the generator $\alpha \in \{a, b, c, d\}$ is $$ defined by
\[ \gamma_\alpha(x) = \left\{\begin{array}{ll}
1 & \mbox{if $x \in U_\alpha$},
-1 & \mbox{if $x \in \sigma(U_\alpha)$},
0 & \mbox{if $x \in V \setminus (U_\alpha \cup \sigma(U_\alpha))}.
\end{array}\right. \]
\end{comment}
where $U_a = [a]_0, U_b = [C]_0 \cup [D]_0, U_c = [B]_0 \cup [D]_0, U_d = [B]_0 \cup [C]_0$. It is trivial to check that this gives an action of $H$. We will also denote this new action by $\act_\jump$, i.e.\ for now we have a system $(\act_\jump, H, \Vee)$.

Suppose now $w \in H$ is presented as a freely reduced word over the generators of $H$, let $\gamma_i : \Vee \to \Z$ be the cocycle of the element $w_{[i, |w|-1]} \in H$, and let $k$ be large enough that $\gamma_i(\Vee) \subset [-k+1, k-1]$ for all $0 \leq i \leq |w| - 1$. Then for all $i$, $w_{[i, |w|-1]} \act_{\jump} (x) = \sigma^\ell(x)$ for some $|\ell| < k$. From this, and the fact $\Vee$ has no periodic points (Lemma~\ref{lem:MinimalNoPeriodic}), it follows by comparing the definitions of the two jump actions that for all $x \in \Vee$, if we let $u = x_{[-k, -1]}, v = x_{[0, k]}$, we have $w \act_{\jump} x = x \iff w \act_{\jump} u \str v = u \str v$. Namely, the movement of the origin in our topological full group action on $\Vee$ exactly mimics the movement of the star in the jump action.

Thus, since the language of $\Vee$ contains only subwords of words $w_n$ for various $n$, $\act_\jump$ is in fact a well-defined action of $\Grig$. Since the jump action of $\Grig$ on the union of starrings of $w_n$ over all $n$ is clearly faithful (because for a fixed $n$ this is conjugate to the tree action on $\{0,1\}^n$) and they all appear in the language of $\Vee$, this is in fact a faithful representation of $\Grig$.
\end{proof}

We define the \emph{reversal} map $f : \Vee \to \Vee$ by $f(x)_i = x_{-1-i}$. This is clearly an involution, and it preserves $\Vee$, since the language of $\Vee$ is obviously invariant under reversal of finite words by Lemma~\ref{lem:Language}. The reversal $f$ also has no fixed points, simply because for all $\alpha \in \{a, B, C, D\}$ the cylinder $[\alpha]_0$ is mapped by $f$ onto $[\alpha]_{-1}$, and these cylinders are disjoint.

\begin{lemma}
\label{lem:VTCover}
The system $(\act_\jump, \Grig, \Vee)$ is an extension of $(\act_\tree, \Grig, \Tree)$ by a factor map $\phi : \Vee \to \Tree$ which is $2$-to-$1$ on a residual set, and $6$-to-$1$ on a countable set of fibers. We also have $\phi \circ f = f \circ \phi$.
\end{lemma}

\begin{proof}
The definition of $\phi$ in a sense arises directly as an inverse limit of the maps $\phi_n$. Let $x \in \Vee$ be arbitrary, and $n \geq 1$. Then by definition of $\Vee$, $x$ is of the form
\[ \dots w_n \beta_{-1} w_n \beta_0 w_n \beta_1 w_n \beta 2 \dots \]
for some symbols $\beta_i \in \{B, C, D\}$. We show by induction on $n$ that for each $x \in V$, this choice of $n$th level decomposition is unique. This is obvious for $n = 1$ as $w_1 = a$. For $n > 1$, we first use the inductive hypothesis to find that there is a unique decomposition
\[ \dots w_{n-1} \beta_{-1} w_{n-1} \beta_0 w_{n-1} \beta_1 w_{n-1} \beta_2 \dots. \]
Any decomposition on level $n$ will give a decomposition on level $n-1$, so we only have two choices left for the $n$th level decomposition.

By definition of $\Vee$, if $w_n = w_{n-1} \alpha w_{n-1}$, one of the period-$2$ subsequences of $\beta_i$ is all-$\alpha$. Say the one containing $\beta_1$ is, so we can write the decomposition of $x$ above as
\[ \dots (w_{n-1} \alpha w_{n-1}) \beta_{-2} (w_{n-1} \alpha w_{n-1}) \beta_0 (w_{n-1} \alpha w_{n-1}) \beta_2 (w_{n-1} \alpha w_{n-1}) \dots \]
If it were possible to find such $x$ where an arbitrarily long subsequence of $b_{2i}$ is also all-$\alpha$, we would find the periodic point $(w_{n-1} \alpha)^\Z$ in $\Vee$. But this is contradicted by the fact $\Vee$ has no aperiodic points.

We call the $w_n$ that appear in the decomposition the \emph{natural $w_n$s}. From the uniqueness, it is automatic that the set of points $x$ where a natural $w_n$ appears in the coordinates $[0, 2^n-1]$ is clopen, and its shifts partition the space into $2^n$ clopen sets.

For $x \in \Vee$, consider now the unique decomposition on level $n$. For any $n \geq 1$, the origin is next to, or inside, a unique natural $w_{n+1}$, in the sense that the leftmost coordinate of the interval where this $w_{n+1}$ appears is in $(-\infty, 0]$, and the rightmost is in $[-1, \infty)$ (this is according to the convention that the origin is considered to be to the left of the position it is ``at'').
We call this unique $w_{n+1}$ the \emph{central $w_{n+1}$}. Say the central $w_{n+1}$ occupies $[i, i+2^{n+1}-2]$, where then $i \in [-2^{n+1} + 1, 0]$. Then define
\[ \psi_n(x) = \phi_{n+1}(-i)_{[0, n-1]}. \]

We want to show that $\psi_{n+1}(x)_{[0,n-1]} = \psi_n(x)_{[0,n-1]}$ for all $n$, and that each $\psi$ is a factor map to the finite system $(\tree, \Grig, \{0,1\}^n)$, as then $\phi(x)_{[0,n-1]} = \psi_n(x)$ defines an equivariant map from $(\Grig, \Vee)$ to $(\Grig, \Tree)$.

% Consider now the origin as giving the starring of a unique natural $w_{n+1}$-word, use the map $\phi_{n+1}$ to get a word of length $n+1$ corresponding to this starring position, and then drop the last bit. We show that this is well-defined, i.e.\ the value of $\phi(x)_{[0,n-1]}$ based on the $(n+1)$th decomposition agrees with the value of $\phi(x)_{[0,n-1]}$ that we obtain from the $(n+2)$th decomposition.

The gist of the proof is that in the construction of the sequence of values $\phi_{n+1}(0), \ldots, \phi_{n+1}(2^{n+1}-1)$, the words in the first $n$ bits follow first the sequence $\phi_n(0), \ldots, \phi_n(2^n-1)$ and then the inverse of it, so the sequence $\phi_{n+1}(0)_{[0,n-1]}, \ldots, \phi_{n+1}(2^{n+1}-1)_{[0,n-1]}$ is a palindrome. Now, in the construction of $\phi_{n+1+i}$ for $i \geq 1$ we will simply keep repeating this palindrome, since for a palindrome $w$ we have the equality $ww^R = ww$.

Let us explain this in more detail. Suppose that the central $w_{n+1}$ is in $[i, i+2^{n+1}-2]$. There are two possible intervals where the central $w_{n+2}$ may appear, namely it may appear to the right (in $[i, i+2^{n+2}-2]$) or to the left (in $[i-2^{n+1}, i+2^{n+1}-2]$) of the central $w_{n+1}$. First suppose the central $w_{n+2}$ is in $[i, i+2^{n+2}-2]$. Then in the relative coordinates of the central $w_{n+2}$, the star position $j = -i$ (corresponding to the origin) is in the first half $[0, 2^{n+1}-1]$. By the definition of $\phi_{n+2}$, we then have $\phi_{n+2}(j) = \phi_{n+1}(j) 1$, and the definitions agree.

If on the other hand the central $w_{n+2}$ is to the left (in $[i-2^{n+1}, i+2^{n+1}-2]$), then the position corresponding to the origin in the central $w_{n+2}$ is $j = 2^{n+1} - i$, which is on the right half (since $i$ is nonpositive). By the definition of $\phi_{n+2}$, we then have $\phi_{n+2}(j) = \phi_{n+1}(F_{n+2}(j)) 0$, where $F_{n+2}(j) = 2^{n+2} - 1 - j$. We have
\[ F_{n+2}(j) = 2^{n+2} - 1 - j = 2^{n+1} - 1 + i = F_{n+1}(-i), \]
so $\phi_{n+2}(j) = \phi_{n+1}(F_{n+1}(-i)) 0$.

Using the level $n+1$ decomposition we would use instead the value of $\phi_{n+1}(-i)$ (like in the previous paragraph). Thus, it finally suffices to show that the first $n$ coordinates of $\phi_{n+1}(j)$ and $\phi_{n+1}(F_{n+1}(j))$ agree for all $j \in [0, 2^n-1]$. This is indeed immediate from the recursive definition of $\phi_{n+1}$.

So far, we have shown that the map $\phi : \Vee \to \Tree$ is well-defined. Continuity is clear from the definition. For equivariance, we observe that for large enough $t$, the map $\psi_n$ must behave equivariantly on $U_m = \bigcup_{\ell = t}^{m-t} [w_m]_{-\ell}$ for any $m$, because on the configurations where the word $w_m$ defining the cylinder is indeed the central $w_m$ (there are also configurations in $[w_m]_{-\ell}$ where the central $w_m$ is elsewhere), we simply mimic the action on starrings of $w_m$ and thus the $\psi_n$-image must follow the conjugacy $\phi_n$. The union of the $U_m$ is easily seen to be dense, so $\psi_n$ must be equivariant everywhere by continuity. As $\phi$ is just the inverse limit of these $\phi_n$, it is equivariant. Now for surjectivity, it suffices to observe that $(\Grig, \Tree)$ is minimal, which is obvious from Lemma~\ref{lem:ActionOnStarrings}.

Finally, we study the cardinality of fibers. It is easy to see that $\phi(f(x)) = \phi(x)$, since when reversing the configuration, we reflect the central $w_{n+1}$, which, as discussed above, immediately by the recursive definition of $\phi_{n+1}$, preserves the first $n$ bits. On the other hand, if we know all bits of $\phi(x)$, we know the positions of all the natural $w_n$s up to reversal. If this determines the point in $\Vee$ up to reversal, the fiber has cardinality $2$ as claimed, as the action of $f$ is free.

Suppose then that the point is not determined up to reversal. This means that at least one position is not inside a natural $w_n$ for any $n$. By shifting, we may assume this position is the origin. Up to reversal, we may assume the central $w_n$s are always to the left. Then the central $w_n$s are nested, and in fact we have
\[ x_{[0, \infty)} = \alpha \lim_n w_n, \]
\[ x_{(-\infty, -1]} = \lim_n w_n. \]
for some $\alpha \in \{B, C, D\}$, where the limits are taken in the obvious sense (to the right, and to the left, respectively). By Lemma~\ref{lem:Language}, each of the three cases can occur for $\alpha$, and this then gives the $6$-to-$1$ orbits.
\end{proof}

% The points in the $6$-to-$1$ orbit are called \emph{singular}. TODO: Clarify this. Historically singulars have had defs (1) point either moved or fixes a nbhd, 2) continuity points of map to stabilizer, 3) quotient of stabilizer by nbhd stabilizer, i.e.\ something about germs.

% We will call the points with $6$-to-$1$ fibers \emph{singular points}, and the remaining symbol between the ``$w_{\omega}$''s is called the \emph{special symbol}.

%\begin{definition}
%\emph{System $V$} is the $\Grig$-system $(\Grig, V)$, where $\Grig$ acts by $\act_{\jump}$.
%\end{definition}

\subsection{Action on Schreier graphs}

The system studied by Vorobets in \cite{Vo12} is essentially the same as $(\Grig, \Vee)$, except that we do not keep track of the orientation of the line (Vorobets studies this in the setting of marked Schreier graphs).

\begin{definition}
Let $f : \Vee \to \Vee$ be the reversal map $f(x)_i = x_{-1-i}$. Then we define $\Voro = \Vee/f$ as the topological quotient $\{\{x, f(x)\} \;|\; x \in \Vee\}$, and $\Grig$ acts on $\Voro$ by $g \act \{x, f(x)\} = \{g \act_\jump x, g \act_\jump f(x)\}$.
\end{definition}

Note that $g \act_\jump f(x) = f(g \act_\jump x)$ by a direct computation (the definition of the jump action is obviously reversal-symmetric), so $\{g \act_\jump x, g \act_\jump f(x)\} = \{g \act_\jump x, f(g \act_\jump x)\} \in \Voro$ and $\Voro$ is indeed closed under the action.

\begin{lemma}
\label{lem:TANF}
The system $(\act,\Grig,\Voro)$ is totally absolutely non-free.
\end{lemma}

\begin{proof}
Let $x \in \Vee$ be arbitrary. We claim that the symbols in $x$ can be uniquely deduced from its stabilizer under the action of $\act_\jump$, up to reversal. To see this, observe that up to reversal we may assume $x_{-1} = a$. Then the symbol $x_0$ at the origin can be deduced from the information of which of the generators $b, c, d$ fix $x$. Then we can inductively apply this procedure to $a \act_\jump x$ and any nontrivial shift $b \act_\jump  x, c\act_\jump  x, d\act_\jump x$, as from the stabilizer $H$ of $x$, we also know the stabilizer $gHg^{-1}$ of $g\act_\jump x$. Since $\Voro = \Vee/f$, we have shown that the point $\{x, f(x)\} \in \Voro$ is determined by its stabilizer.
\end{proof}

The following is roughly (part of) the Factor Theorem of \cite{GrLeNa17}, and we only sketch the proof. See \cite{Vo12} for the definition of the space of marked Schreier graphs.

\begin{lemma}
The system $(\Grig, \Voro)$ is topologically conjugate to the system of Schreier graphs studied in \cite{Vo12}.
\end{lemma}

\begin{proof}
Recall that Vorobets' system in \cite{Vo12} is obtained from computing the Schreier graph of $1^\omega$ under the tree action, computing the system of marked Schreier graphs generated by it, and finally removing the markings of the original Schreier graph (which are isolated points).

The map that takes $x \in \Vee$ to its Schreier graph is easily seen to be continuous. By the previous lemma, this map has kernel pair precisely $\{\{x, f(x)\} \;|\; x \in \Vee\}$. In the setting of compact Hausdorff spaces, continuous surjections are quotient maps, and two quotients with the same kernel pair are homeomorphic, so $\Voro/f$ is indeed a system of Schreier graphs.

Since both systems are minimal, and both contain points that are arbitrarily good approximations of finite marked Schreier graphs arising from the SAWs $w_n \str \alpha w_n$, they must be the same system of marked  Schreier graphs.
\end{proof}

\section{Basic constructions on SFTs}
\label{sec:BasicConstructions}

\subsection{The union lemma}

In this section, we show that class of SFTs is closed under disjoint union on finitely-generated groups. This is well-known at least on the groups $\Z^d$ (see for example \cite[Proposition~3.4.1]{LiMa21} for $\Z$), and the proof for a general group differs mostly in notation.

We note that if $P_1, \ldots, P_n$ are the forbidden patterns defining an SFT, we may assume all the $P_i$ have the same domain $D$, more generally we can always extend the domain of a forbidden pattern without changing the set of configurations it forbids, simply by replacing it by all its possible extensions over the alphabet.

\begin{lemma}
Let $X_1, \ldots, X_k \subset A^G$ be disjoint SFTs, and suppose $G$ is finitely generated. Then the finite union $X = \bigcup_i X_i$ is an SFT.
\end{lemma}

\begin{proof}
Subshifts are closed sets in $A^G$ because they are defined by forbidding open sets (namely sets of the form $g^{-1} [P]$). Thus they are compact, and by basic topology there exists $\epsilon > 0$ such that the distance between any $x \in X_i$ and $y\in X_j$ is at least $\epsilon$ whenever $i \neq j$. By the definition of the product topology of $A^G$, and the compatibility of the metric with it, there exists a finite set $D \subset G$ such that $d(x, y) \geq \epsilon \implies x|_D \neq y|_D$. By possibly increasing $D$, we may assume that all the SFTs $X_i$ admit defining sets of forbidden patterns with shape $D$. List the forbidden patterns of $X_i$ as $P_{i,1}, \ldots, P_{i,k_i} : D \to A$.

We give a set of forbidden patterns for an SFT $Y$, and then show that $Y = X$. Let $S \subset G$ be a finite generating set. The alphabet of $Y$ is taken to be $A$. As for forbidden patterns, first in $Y$ we forbid each pattern $P : D \to A$ such that $P$ does not appear in any of the $X_i$. Then for each $s \in S$, we forbid the pattern $P : D \cup Ds \to A$ if for some $i \neq j$, we have
\[ \exists x \in X_i: x|_D = P|_D \wedge \exists y \in X_j: y|_{Ds} = P|_{Ds}. \]

First let us show $X \subset Y$. Let $x \in X$; we show that none of the forbidden patterns appear in $x$ at $1_G$ (this suffices because if one were to appear at $g$, then it appears in $gx$ at $1_G$, and $gx \in X$ because $X$ is shift-invariant as a union of shift-invariant sets). We have $x \in X_i$ for some $i$. By definition, $x|_D$ appears in $X_i$, so it is not in the first set of forbidden patterns.

Consider then the second family of forbidden patterns, i.e.\ pick some $s \in S$, and one of the forbidden patterns $P : D \cup Ds \to A$. Suppose that the SFTs from where the $D$- and $Ds$-patterns are taken have distinct indices $i', j$, i.e.\ $\exists x' \in X_{i'}: x'|_D = P|_D$ and $\exists y \in X_j: y|_{Ds} = P|_{Ds}$. Suppose now that $P$ appears in $x$ at $1_G$. Since $d(x, y) > \epsilon$ for any point of $y \in X_k$ with $k \neq i$, and $x|_D = P|_D$, we must actually have $i' = i$, as the $D$-shaped pattern $P|_D$ appears in both $X_{i'}$ and $X_i$.

The assumption that $x|_{D \cup Ds} = P$ now implies $\exists y \in X_j: y|_{Ds} = x|_{Ds}$. But now $sx \in X_i$ and $sy \in X_j$ since SFTs are shift-invariant. For $d \in D$ we have $sy_d = y_{ds}$ and $sx_d = x_{ds}$ so $sx|_D = sy|_D$, again contradicting the fact that no $D$-shaped pattern appears in both $X_i$ and $X_j$ for $i \neq j$. We conclude that $P$ cannot appear in $x$. This concludes the proof that $X \subset Y$.

Now let us prove $Y \subset X$. Suppose $y \in Y$. Then $y$ avoids the first set of forbidden patterns so for some $i$ we have $y|_D = x|_D$ for some $x \in X_i$. We claim that then for all $g \in G$ we have that $gy|_D$ appears in $X_i$. It suffices to show this for generators, i.e.\ whenever $y|_D$ appears in $X_i$, so does $sy|_D$, for $s \in S$. Suppose this fails, and for some $y \in Y$, we have that $y|_D$ appears in $X_i$ but $sy|_D$ does not. We have $sy \in Y$ because $Y$ is an SFT, and thus because $sy$ avoids the first set of forbidden patterns we have that $sy|_D$ appears in some $z \in X_j$. This means that for all $d \in D$,
\[ y_{ds} = sy_d = z_d = s^{-1}z_{ds}, \]
so $y|_{Ds} = s^{-1}z|_{Ds}$. Note that $s^{-1}z \in X_j$ by shift-invariance of $X_j$.

Define $P = y|_{D \cup Ds}$. Now $y|_D$ appears in $X_i$, and $P|_{Ds} = y|_{Ds} = s^{-1}z|_{Ds}$ appears in $X_j$, so $P$ is one of our forbidden patterns. This contradicts the aassumption on $y$. We conclude that indeed if $y \in Y$ and $y|_D$ appears in $X_i$, then the same is true for all shifts of $y$.

But this means that whenever $y|_D$ appears in $X_i$, actually $y \in X_i$, because if $gy|_D$ always appears in a configuration of $X_i$, in particular it is not a forbidden pattern for $X_i$. We conclude that every configuration $y \in Y$ belongs to one of the $X_i$, thus $Y \subset X$.
\end{proof}

\begin{remark}
If the effective alphabets (the sets of symbols that actually appear in configurations) $A_i \subset A$ of the $X_i$ are disjoint, the proof is simplified a bit, as we don't need to use the large shapes $D \cup Ds$ but can simply require each configuration to be over a single alphabet with patterns of shapes $\{1_G, s\}$. In fact, we can get an alternative proof of the above lemma from this, by taking a ``higher block presentation'' for each $X_i$ by having each node remember the pattern around it. A higher block presentation is clearly a topological conjugacy, and SFTs are closed under topological conjugacy, so the union turns into a symbol-disjoint union. Finally you can recode it back, by dropping the extra information in each node, and we get that the union is SFT.
\end{remark}

\begin{remark}
The above lemma is sharp in the sense that if $G$ is not finitely-generated, then the union of finitely many (at least two) disjoint SFTs is never an SFT. This can be proved analogously to \cite[Lemma~1]{Sa18}, replacing $0$- and $1$-patches by patterns from two of the disjoint SFTs.
\end{remark}

\begin{lemma}
\label{lem:SoficUnion}
A finite (not necessarily disjoint) union of sofic shifts on a finitely generated group $G$ is sofic. More generally, for finitely generated $G$, a finite union of (not necessarily disjoint) SFT covered $G$-systems is SFT covered.
\end{lemma}

\begin{proof}
Suppose $X_1, \ldots, X_k$ are systems with SFT covers $\phi_i : Y_i \to X_i$. We may assume the alphabets of the $Y_i$ are disjoint. Then $Y = \bigsqcup_i Y_i$ is SFT, and the map $\phi : Y \to \bigcup_i X_i$ defined by $\phi(y) = \phi_i(y)$ for $y \in Y_i$ is clearly shift-invariant, continuous and surjective. This implies the ``more generally'' claim. In the sofic case, we simply observe that the union of the $X_i$ is a subshift.\footnote{Note that in general, the union of two expansive subsystems of a fixed system is \emph{not} expansive, here instead we are saying that the union is a subsystem of an expansive system (the full shift), and every subsystem of an expansive system is expansive.}
\end{proof}

\begin{comment}
Proof that finite union of two expansive systems X, Y if expansive:
Suppose \epsilon is expansive constant for both.
If x, y \in X, then we can \epsilon-separate in X.
If x, y \in Y, then we can \epsilon-separate in Y.
If x \in X, y \in Y, then if distance of x, y is less than \epsilon/2,
take x' \in X at distance .
\end{comment}

\subsection{The finite-index lemma}

We now show that the existence of SFT covers is preserved under finite group extensions, in a rather general sense. We are not aware of a reference for such a result, but closely related commensurability-closure results can be found in the literature, see e.g.\ \cite[Proposition~1]{DrSc07} or \cite[Lemma~7.2]{BaSaSa21}.

\begin{lemma}
\label{lem:FiniteIndex}
Let $G$ be a finitely-generated group and let $(\act, G, X)$ be an arbitrary zero-dimensional system. Suppose $H$ is a finite-index subgroup of $G$, and let $X' \subset X$ be an $H$-invariant subset (i.e.\ the restriction $(\act, H, X')$ is well-defined), and $G \act X' = X$. Suppose the subaction $(\act, H, X')$ admits an SFT cover. Then $(\act, G, X)$ admits an SFT cover.
\end{lemma}

In our application we will have $X' = X$.

\begin{proof}
Pick coset representatives $R \subset G$ for $H$, i.e. $G = RH$. We may assume $1_G \in R$. Let $Y \subset W^H$ be an SFT cover for $(H, X')$, let $\phi : Y \to X'$ be the covering map. We may suppose this is given by Wang tiles, so we have a generating set $S$ for $H$, $W \subset C^{S^{\pm}}$ (for each $s \in S$, including involutions, we have separate positive and negative copies $s^+$ and $s^-$) and the SFT rule is the color-matching rule: $x \in Y$ if and only if for all $g \in H$, the pair $(x_g, x_{sg})$ satisfies the $s$-color-matching rule $(x_g)_{s^+} = (x_{sg})_{s^-}$.

We take as the new alphabet $Q = \{\bot\} \cup W$ and define an SFT $Z \subset Q^G$ as follows: The first SFT rule is that for each $s \in S$, we require that if $x_g \in W$, then $x_{sg} \in W$, and the pair $(x_g, x_{gs})$ satisfies the $s$-color-matching rule. The second SFT rule is that if $x_g \in W$, then $x_{rg} = \bot$ for all $r \in R \setminus \{1_G\}$. Finally, require that for each $g \in G$, at least one symbol in $x|_{R^{-1}g}$ is in $W$. These three rules define an SFT $Z$.

We claim that $Z$ is nonempty. Indeed, if $y \in Y$ then define $x = f(y) \in Q^G$ by $x|_H = y$, and $x_g = \bot$ otherwise. Then the first rule is obviously satisfied. The second is satisfied because $R$ is a set of coset representatives for $H$, so
\[ x_g \in W \implies g \in H \implies x_{rg} \notin H \implies x_{rg} = \bot \]
for $r \in R \setminus \{1_G\}$. Finally, the third rule is satisfied because if $g \in G$ then $g = rh$ for some $r \in R$, $h \in H$, and then $r^{-1}g = h \in H \implies x_{r^{-1}g} \in W$.

If $x \in Z$ and $x_{1_G} \in W$, then it is easy to see that we have $\{g \in G\;|\; x_g \in W\} = H$, namely the set on the left contains $H$ by the first SFT rule, and since $G = RH$ the second rule forces $x_g = \bot$ for $g \notin H$. Consider now a general $z \in Z$. By the third rule, at least one symbol in $z$ is in $W$, indeed there exists $r \in R$ such that $r^{-1}z_{1_G} \in W$. Since $x = r^{-1}z$ satisfies $x_{1_G} \in W$, by the first observation of this paragraph we have $\{g \in G\;|\; x_g \in W\} = H$, and then for $rx = z$ we have $\{g \in G\;|\; z_g \in W\} = Hr^{-1}$. Since $R^{-1}$ is a set of right coset representatives, we conclude that for any $z \in Z$, there in fact exists a unique $r \in R$ such that $r^{-1}z|_{1_G} \in W$. We say $r \in R$ is the \emph{phase} of $z$.

Next, we describe a function $\psi : Z \to X$. Let $x \in Z$. If $x_{1_G} \in W$, i.e.\ the phase is $1_G$, then by the first SFT rule the configuration $x|_H \in W^H$ is in $Y$, and we define $\psi(x) = \phi(x|_H) \in X$ in this case (the image is in $X'$ but we see it as an element of $X$). %In general, by the second SFT rule, exactly one of the symbols in $x|_{R^{-1}}$ is in $W$, equivalently exactly one of the elements $R^{-1} \cdot x$ has an element of $W$ at the origin.
If the phase is $r \in R \setminus \{1_G\}$ meaning $(r^{-1} x)|_H \in Y$ then we define $\psi(x) = r \act \phi((r^{-1} x)|_H)$.

It is obvious that $\psi$ is continuous, because it is defined piecewise on $|R|$ many disjoint clopen sets, and on the clopen set corresponding to $r \in R$ it is defined by the formula $\psi(x) = r \act \phi((r^{-1} x)|_H)$, which is a composition of finitely many continuous functions. Namely, the outer action $y \mapsto r \act y$ was assumed continuous, the inner (shift) action $x \mapsto r^{-1} x$ is                                                                                                                                                                                                                                                                                                                                                                                                                                                                                                                                                                                                                                                                                                                                                                                                                                                                                                                                                                                                                                                                                                                                                                                                                                                                                                                                                                                                                                                                                                                                                                                                                                                                                                                                                                                                                                                                                                                                                                                                                                                                                                                                                                                                                                                                                                                                                                                                                                                                                                                                                                                                                                                                                                                                                                                                                                                                                                                                                                                                                                                                                                                                 continuous, and $y \mapsto \phi(y|_H)$ is continuous by the assumption that $\phi$ is a factor map and restriction is continuous.

We claim that $\psi$ is surjective onto $X$. To see this, let $x \in X$ be arbitrary. If $x \in X'$, pick $y \in Y$ a $\phi$-preimage, and define a configuration $z = f(y) \in Z$ as above. Now by definition $\psi(z) = x$. % $z|_H = y$ and $z_g = \bot$ for $g \notin H$. It is easy to see that $z \in Z$: For $g \in G$, exactly one of $z_{Rg}$ is in $H$ (note that $R$ permutes the left cosets because $H$ is normal), and thus exactly one of these elements is in $W$ so the second SFT rule is respected; on the other hand, on $H$ where the symbols are in $W$, we respect the first SFT rule because $z|_H$ is taken from $Y$. Now by definition $\psi(z) = x$.

Next, suppose $x \notin X'$. By the assumption $G \act X' = X$, we have $x = g \act x'$ for some $g \in G, x' \in X'$, and because $H \act X' = X'$, we may assume $r \act x' = x$ for $r \in R$ (up to possibly changing $x'$). Let $\phi(y) = x'$ and again take $z = f(y)$ so $\psi(z) = x'$. Now consider the point $z' = r z$. Clearly its phase is $r$, since $z_{1_G} \in W$. Thus 
\[ \psi(z') = r \act \phi((r^{-1} z')|_H) = r \act \phi(z)|_H = r \act x' = x. \]

\begin{comment}
Since $R$ permutes the left cosets of $H$, there exists a unique $r' \in R$ such that $r'r \in H$, again since $R$ permutes the left cosets. Suppose $r'r = h \in H$. Then $(r' \act (r \act z))_{1_G} = (h \act z)_{1_G} = z_h \in W$ (since $z|_H \in Y$), and by the definition of $\psi$, we have
\begin{align*}
\psi(r \act z) &= r'^{-1} \act \phi((r'r \act z)|_H) \\
&= rh^{-1} \act \phi((h \act z)|_H) \\
&= rh^{-1} \act h \act \phi(y)) \\
&= r \act \phi(y) = rx' = x.
\end{align*}
(Note that we used $h = r'r \implies h^{-1} = r^{-1}r'^{-1} \implies rh^{-1} = r'^{-1}$.)
\end{comment}

Next, we claim that $\psi$ intertwines the shift map and the action of $G$ i.e.\ $\psi(gz) = g \act \psi(z)$ for all $z \in Z, g \in G$. Suppose $\psi(z) = x$ and let $g \in G$. The idea is to go from $z$ to $gz$ by adding an artificial shift step changing the phase to $1_G$. The idea is that we visualize a configuration as a ``comb'' having its $W$-symbols on a ``spine'' of shape $Hr^{-1}$, and other $g \in G$ on the ``teeth'' reached from the spine by moving by elements of $R$ on the left. For movement on the spine, the fact the actions are conjugated is obvious, and the definition of $\psi$ outside the spine is explicitly the desired conjugation formula.

We now translate the idea into formulas. %We first consider the steps where we move on teeth and along the spine.
First, to move along teeth to the spine, suppose $z \in Z$ has phase $r$, and $g = r^{-1}$. Then in fact the definition of $\psi$ gives directly $\psi(z) = r \act \phi(g z)$ meaning $g \act \psi(z) = \phi(g z) = \psi(g z)$ (since $gz$ has phase $1$ so $\phi(gz) = \psi(gz)$). Moving from the spine up the teeth (i.e.\ the case that $z \in Z$ has phase $1$ and $g \in R$) is given by a similar calculation. As for moving on the spine, suppose $z$ and $gz$ have phase $1_G$. Then $g \in H$ and by definition of the shift map and the assumption that $\phi$ commutes with the $H$-actions, we have
\[ g \act \psi(z) = g \act \phi(z|_H) = \phi(gz|_H) = \psi(gz). \]

%phase is r means Hr^{-1} has W-symbols, equivalently r^{-1}z has W-symbols in H.
%It suffices to show that \phi(gz) = g \act \phi(z) holds in the following cases:
%* z has phase r, and g = r^{-1}
%* z has phase 1_G, and g \in H.
%* z has phase 1_G, gz has phase r

%The point of the formulas is just that by construction we have correct intertwining when moving along the ``comb'' where we can use one of the $H$-cosets as a spine and the $R$-translates as ``teeth'', and these movements connect the group

In the general case, let $r \in R$ be the phase of $z$, and let $r'$ be the phase of $gz$, so that $r^{-1}z|_H, r'^{-1}gz|_H \in Y$, i.e.\ $z$ has $W$-symbols precisely in the set $Hr^{-1}$, and on the other hand precisely in the set $Hr'^{-1}g$ (recall that there is a unique left coset containing $W$-symbols). We then have $Hr^{-1} = Hr'^{-1}g$ so $h = r'^{-1}gr \in H$, and we have $g = r'hr^{-1}$.

We first move to the spine:
\[ \psi(r^{-1}z) = r^{-1} \act \psi(z). \]
Next, $r^{-1}z$ has phase $1_G$ and $r'^{-1}gr \in H$ so moving along the spine gives:
\[ \psi(r'^{-1}gr \cdot r^{-1}z) = r'^{-1}gr \act \psi(r^{-1}z). \]
Next, $r'^{-1}gr \act r^{-1}z$ has phase $1_G$ (since translation by an $H$-element clearly preserves phase $1_G$), so moving away from a spine gives:
\[ \psi(r' \cdot r'^{-1}grr^{-1}z) = r' \act \phi(r'^{-1}grr^{-1}z). \]
All in all, 
\begin{align*}
\psi(g \cdot z) &= \psi(r' \cdot r'^{-1}grr^{-1} z) \\
&= r' \act \psi(r'^{-1}gr \cdot r^{-1}z) \\
&= r' \act r'^{-1}gr \act \psi(r^{-1}z) \\
&= r' \act r'^{-1}gr \act r^{-1} \act \psi(z) \\
&= g \act \psi(z).
\end{align*}
\end{proof}

\section{The system $(\Grig, \Vee)$ is a subshift}
\label{sec:VSubshift}

\begin{lemma}
\label{lem:TFGExpansive}
Let $(G, X)$ be a compact zero-dimensional dynamical system. If the topological full group $\llbracket (G, X) \rrbracket$ has a finitely-generated expansive subgroup, then $(G, X)$ is expansive.
\end{lemma}

\begin{proof}
Let $H = \langle h_1, \ldots, h_k \rangle \leq \llbracket (G, X) \rrbracket$ be expansive with constant $\epsilon > 0$, and suppose the generating set $\{h_i\}$ is symmetric. Since $h_i$ is in the topological full group, there is a continuous cocycle $\gamma_i : X \to G$ for it, and since $G$ is discrete, this only depends on a clopen partition of $X$, and $\gamma_i(X) = F_i \Subset G$ is finite. By possibly decreasing $\epsilon$ we may assume $d(x, y) < \epsilon \implies \forall i: \gamma_i(x) = \gamma_i(y)$.

Suppose now $x \neq y$. By expansivity of the subgroup $H$, there exists  $h = h_{i_m} \cdots h_{i_2} h_{i_1} \in H$ such that we have $d(h x, h y) \geq \epsilon$.

Take a minimal possible $m$. By definition we have
\[ hx = g_{i_m} \cdots g_{i_2} g_{i_1} x \]
where $g_{i_j} = \gamma_{i_j}(h_{i_{j-1}} \cdots h_{i_2} h_{i_1} x)$. By the assumption on $\epsilon$ and minimality of $m$, we also have $hy = g_{i_m} \cdots g_{i_2} g_{i_1} y$, thus $d(gx, gy) \geq \epsilon$ for $g = g_{i_m} \cdots g_{i_2} g_{i_1}$, proving expansivity.
\end{proof}

\begin{lemma}
Let $(\act_\jump, \Grig, \Vee)$ be the group $\Grig$ acting on $\Vee$ by the jump action. Then $\sigma_\Vee \in \llbracket (\act_\jump, \Grig, \Vee) \rrbracket$, where $\sigma_\Vee$ is the shift map on $\Vee$ defined by $\sigma_\Vee(x)_i = x_{i+1}$.
\end{lemma}

\begin{proof}
Take the partition $A_1 = [a]_0$, $A_2 = [B]_0$, $A_3 = [C]_0$, $A_4 = [D]_0$, and $g_1 = a$, $g_2 = c$, $g_3 = d$, $g_4 = b$. This defines the shift map, because by the definition of the jump action on $\Vee$, $g_i$ shifts points in $[A_i]_0$ one step to the left.
\end{proof}

\begin{lemma}
\label{lem:IsSubshift}
The system $(\Grig, \Vee)$ is expansive, hence conjugate to a subshift over some finite alphabet.
\end{lemma}

\begin{proof}
The space is clearly a Cantor set, so it suffices to show expansivity. By the previous lemma, $\llbracket (\Grig, \Vee) \rrbracket$ contains the finitely-generated subgroup $\langle \sigma_\Vee \rangle \cong \Z$, which already acts expansively. By Lemma~\ref{lem:TFGExpansive}, $(\Grig, \Vee)$ is expansive.
\end{proof}

\begin{remark}
Note that $\sigma_\Vee$ is not itself the jump action of any element of the group $\Grig$, since $\Grig$ is a torsion-free. %See \cite{Vo20} .
\end{remark}

% The topological full group of $\Vee$, as a $\Z$-subshift, has been completely described by Vorobets \cite{Vo20}: the group $\Grig$ is generated by the cocycles $a \cong \delta_{[a]_0}, b \cong \delta_{[BC]_0}, c \cong \delta_{[BD]_0}, d \cong \delta_{[CD]_0}$. With the related formulas $a \cong \delta_{[a]_0}, B \cong \delta_{[B]_0}, C \cong \delta_{[C]_0}, D \cong \delta_{[D]_0}$ we obtain a group that is usually called the \emph{overgroup} $\overgroup$. %(and this is obviously compatible with the action of $\Grig$ seen as a subgroup). The shift map is not even in the overgroup, but essentially everything else is, in that the topological full group of $\Vee$ is just the direct product $\langle \sigma \rangle \times \overgroup$.
%\end{remark}

\section{The system $(\Grig, \Vee)$ is a sofic subshift}
\label{sec:VSofic}

In this section, we prove the first part of Theorem~\ref{thm:Main} (soficness). The second part (properness) is proved in the next section.

As discussed in the introduction, if $\pi : G \to H$ is a group homomorphism, for $H$-systems we can define their \emph{pullback}, this is just the $G$-system with the same space and with action $g \act x = \pi(g) \act x$. The following theorem is due to Sebasti\'an Barbieri \cite{Ba19}.

\begin{theorem}
\label{thm:ProductSimulation}
Let $G, H, K$ be three finitely-generated infinite groups, and $\pi : G \times H \times K \to G$ the natural projection. Then the $\pi$-pullback of any effective expansive $G$-system admits an $G \times H \times K$-SFT cover.
\end{theorem}

%The proof is sketched in the appendix. The assumption of expansivity here causes us a little bit of extra work. It can in fact be eliminated, as will be shown in [Barbieri-Salo].

For $\alpha, \beta \in \{0,1\}$ let $G_{\alpha \beta}$ be the subgroup of the group $\Grig$ containing those $g \in \Grig$ satisfying:
\[ \forall \alpha', \beta' \in \{0,1\}: \alpha'\beta' \neq \alpha\beta \implies \forall x \in \{0,1\}^\N: g \act_{\tree} \alpha'\beta' x = \alpha' \beta'x. \]
Of course for $g \in G_{\alpha\beta}$ we have
\[ \forall x \in \{0,1\}^\N: \exists y \in \{0,1\}^\N: g \act_{\tree} \alpha \beta x = \alpha \beta y \]
The groups $G_{\alpha \beta}$ are usually called rigid stabilizers of the second level \cite{Gr00}.

The following is well-known \cite{Ha00}:

\begin{lemma}
The subgroup $\Grig_2 = G_{00} \times G_{01} \times G_{10} \times G_{11}$ is of finite index in $\Grig$.
\end{lemma}

Let $\phi : \Vee \to \Tree$ be the map from Lemma~\ref{lem:VTCover}. Consider the clopen set $\Tree_{\alpha\beta} = \alpha\beta\{0,1\}^\N$ of words that begin with $\alpha\beta \in \{0,1\}^2$, and observe that the subgroup $G_{\alpha\beta} \leq \Grig$ stabilizes $\Tree_{\alpha\beta}$ (as a set), while $G_{\alpha'\beta'}$ stabilizes it pointwise. Let now $\Vee_{\alpha\beta} = \phi^{-1}(\Tree_{\alpha\beta})$, another clopen set (since $\phi$ is continuous).

\begin{lemma}
For all $\alpha\beta$, the action of $\Grig_2$ stabilizes $\Vee_{\alpha\beta}$ as a set.
\end{lemma}

\begin{proof}
Clearly $\Grig_2$ stabilizes each $\Tree_{\alpha\beta}$ as a set. If $g \in \Grig_2$ and $x \in \Vee_{\alpha\beta}$, then $\phi(x) \in \Tree_{\alpha\beta}$ so $\phi(g \act_{\jump} x) = g \act_{\tree} \phi(x) \in g \act_{\tree} \Tree_{\alpha\beta} = \Tree_{\alpha\beta}$, thus $g \act_{\jump} x \in \phi^{-1}(\Tree_{\alpha\beta}) = \Vee_{\alpha\beta}$.
\end{proof}

We obtain that $(\Grig_2, \Vee)$ is a direct union of the corresponding subsystems.

\begin{lemma}
\label{lem:IsUnion}
The system $(\Grig_2, \Vee)$ is a disjoint union $\bigsqcup_{\alpha, \beta} (\Grig_2, \Vee_{\alpha\beta})$.
\end{lemma}

Since $\phi$ is a factor map, $G_{\alpha\beta}$ fixes fibers of points in $\Vee_{\alpha'\beta'}$ when $\alpha\beta \neq \alpha'\beta'$, but in principle it could move points within the fibers. We show that this action is trivial.

\begin{lemma}
\label{lem:PointwiseInV}
For all $\alpha\beta \neq \alpha'\beta'$, the action of $G_{\alpha\beta}$ stabilizes $\Vee_{\alpha'\beta'}$ pointwise.
\end{lemma}

%\begin{lemma}
%\label{lem:ActionOnStarrings}
%For each $n$, the Grigorchuk group admits a well-defined action on the set %of starrings of $w_n$ by $\jump$. Furthermore, for a fixed $n$, this action %is conjugate to its tree action on $\{0,1\}^n$.
%\end{lemma}

\begin{proof}
If not, then there is $g \in G_{\alpha\beta}$ that maps some $x \in \Vee_{\alpha'\beta'}$ into a different point in its fiber. Since the action of $g$ is continuous and $\Vee$ is a Cantor space, we may assume $x$ is not in the countable set $C$ of points with fiber size $6$. Thus, $g$ in fact acts as the reversal map $f$ on the point $x$. In particular, $g$ acts nontrivially on every point sufficiently close to $x$, and we conclude that $g$ acts as precisely $f$ on $U \setminus C$, where $U$ is an open set containing $x$. Then it acts as $f$ on all of $U$, since $U \setminus C$ is dense in $U$.

On the other hand, if $U$ is small enough, then $g$ acts as a fixed power of the shift in $U$, since the jump action is by elements of the topological full group. Thus its suffices to show that there cannot be any nonempty open set where $\sigma^k$ acts as the reversal $f$. %Since $f$ acts freely, we cannot have $k = 0$. On the other hand if $\sigma^k(x) = f(x)$ for nonzero $k$, then $\sigma^{2k}(x) = \sigma^k(f(x)) = f(\sigma^{-k}(x))$

To see that this is impossible, observe that any point $x \in \Vee$ in such an open set $U$ would have a shift by some $\sigma^m$ which is again in $U$ (since $(\sigma_\Vee, \Vee)$ is minimal), and the first shift $\sigma^k$ would again give its reflection. Composing two distinct reflections gives a shift, so this would give us a periodic point.

In formulas, suppose $x \in U$. Then $x$ satisfies $x_{i+k} = \sigma^k(x)_i = f(x)_i = x_{-1-i}$, for all $i$. Since $(\sigma_\Vee, \Vee)$ is minimal, there exists a positive $m$ such that $\sigma^m(x) \in U$, so that also $x_{i+k+m} = \sigma^{k+m}(x)_i = f(\sigma^m(x))_i = x_{-1-i+m}$, for all $i$. Then $x_{i+k} = x_{-1-i}$ and $x_{i+k+m} = x_{-1-i+m}$ for all $i$. Replacing $i$ by $i+m$ in the second equality, we conclude $x_{i+k+2m} = x_{i+k}$ for all $i$, so $x$ is periodic, contradicting the aperiodicity of $\Vee$.
%It is clear (e.g.\ from Lemma~\ref{lem:Language}) that for every $m$. 
% every nonempty open set $U$ contains a point whose reversal is not equal to its $m$-shift.
%The action of $G_{\alpha\beta}$ on starrings of $w_n$ is conjugate to the tree action on $\{0,1\}^n$. Since the latter action fixes every word beginning with $\alpha'\beta'$, $G_{\alpha\beta}$ fixes those starrings of $w_n$ that map to  $\alpha'\beta'$.
%that correspond to a Lemma~\ref{lem:ActionOnStarrings},
% Starrings of words $w_n$, seen as positioned words with origin at $\str$, are dense in $V$ in the appropriate two-sided sense. Thus, $G_{\alpha\beta}$ much act trivially on $V$.
\end{proof}

Next, we have to prove expansivity of $(G_{\alpha\beta}, \Vee_{\alpha\beta})$, to be able to apply Theorem~\ref{thm:ProductSimulation}. %As noted, this theorem is true without the expansivity assumption; readers that believe this claim may skip the following lemma.

\begin{lemma}
\label{lem:IsExpansive}
The system $(G_{\alpha\beta}, \Vee_{\alpha\beta})$ is expansive.
\end{lemma}

\begin{proof}
The subgroup $\Grig_2$ is of finite index in $\Grig$, let $\Grig = S \Grig_2$ for a finite set $S$. The group $\Grig$ acts expansively on $\Vee$ by Lemma~\ref{lem:IsSubshift}, so there exists $\epsilon > 0$ such that for any distinct $x, y \in \Vee_{\alpha\beta}$ there exists $g \in \Grig$ such that $d(gx, gy) > \epsilon$. Then for suitable $\delta > 0$, chosen using continuity of the actions of $s \in S$, we still have $d(sgx, sgy) > \delta$ where $s \in S$ is such that $sg \in \Grig_2$. Using the decomposition $\Grig_2 = G_{00} \times G_{01} \times G_{10} \times G_{11}$ we can write $sg = (g_{00}, g_{01}, g_{10}, g_{11})$. For $\alpha'\beta' \neq \alpha\beta$, the element $g_{\alpha'\beta'}$ fixes $x$ and $y$ (since it fixes $\Vee_{\alpha'\beta'}$ pointwise), thus $d(g_{\alpha\beta}x, g_{\alpha\beta}y) = d(sgx, sgy) > \delta$, showing expansivity.
\end{proof}

\begin{lemma}
\label{lem:IsPullBack}
For all $\alpha\beta$, the system $(\Grig_2, \Vee_{\alpha\beta})$ is the pullback of the system $(G_{\alpha\beta}, \Vee_{\alpha\beta})$ under the natural projection $\pi : \Grig_2 \to G_{\alpha\beta}$.
\end{lemma}

\begin{proof}
Using the decomposition $\Grig_2 = G_{00} \times G_{01} \times G_{10} \times G_{11}$, if $x \in \Vee_{\alpha\beta}$ we have
$(g_{00}, g_{01}, g_{10}, g_{11}) \act x = g_{\alpha\beta} \act x$, which is the definition of the pullback.
\end{proof}

\begin{lemma}
For any $\alpha,\beta \in \{0,1\}$, the system $(G_{\alpha\beta}, \Vee_{\alpha\beta})$ is effective.
\end{lemma}

\begin{proof}
We first show that the subset $\Vee_{\alpha\beta} \subset \{a, B, C, D\}^\Z$ is effectively closed (after bijecting $\Z$ with $\N$ by any computable formula). %Strictly speaking, we should biject $\Z$ with $\N$, but this is trivial and only makes notation more complicated, so we will directly think
First we show that $\Vee$ is, i.e.\ a set of forbidden patterns can be enumerated for it. First this, it suffices to give an algorithm for checking whether a given word $u$ appears in the language of the subshift. But Lemma~\ref{lem:Language} immediately gives such an algorithm since it suffices to check whether $u$ appears as a subword of $w_{n+3}$ where $|u| \leq |w_n|$. We now observe that $\Vee_{\alpha\beta}$ is a clopen subset of $\Vee$, so it suffices to add a finite set of additional forbidden patterns to define it. %simply observe that, almost immediately from the definition, $u$ appears in the language if and only if it appears in one of the words $w_n$; and then it is easy to see that if $|u| \leq |w_n|$, then $u$ appears in the language if and only if $u$ is a subword of $w_n \gamma w_n$, for some $\gamma \in \{B, C, D\}$.

Next, we need to show that the actions of elements of $G_{\alpha\beta}$ are computable. Again, we start by showing that even the action of $\Grig$ on $\Vee$ is computable. This is immediate, as we defined the action by an explicit formula: For a generator $g \in \{a, b, c, d\}$, the action on a point $x \in \Vee$ is to look at the symbols in the immediate neighborhood of the origin $0 \in \Z$ in $x$, and then shift the entire configuration. This is conjugated to a computable permutation of $\N$ through our choice of bijection $\Z \cong \N$ in the previous paragraph. Composition of computable functions is computable, and finally the restriction to $(G_{\alpha\beta}, \Vee_{\alpha\beta})$ preserves computability of the actions of all elements $g \in G_{\alpha\beta}$.
\end{proof}

\begin{lemma}
\label{lem:PiecesSofic}
For all $\alpha\beta$, the system $(\Grig_2, \Vee_{\alpha\beta})$ is topologically conjugate to a sofic shift.
\end{lemma}

\begin{proof}
Let $G = G_{\alpha\beta}, H = G_{(1-\alpha)\beta}, K = G_{(1-\alpha)(1-\beta)} \times G_{\alpha(1-\beta)}$. Then $(\Grig_2, \Vee_{\alpha\beta}) = (G \times H \times K, \Vee_{\alpha\beta})$ is the pullback of $(G, \Vee_{\alpha\beta})$ under the natural projection by Lemma~\ref{lem:IsPullBack}. By Lemma~\ref{lem:IsExpansive}, $(G, \Vee_{\alpha\beta})$ is expansive, and by the previous lemma, it is also effective. Thus, Theorem~\ref{thm:ProductSimulation} implies that this system is SFT covered.
\end{proof}

\begin{lemma}
The system $(\Grig_2, \Vee)$ is topologically conjugate to a sofic shift.
\end{lemma}

\begin{proof}
By Lemma~\ref{lem:IsUnion}, $(\Grig_2, \Vee) \cong \bigsqcup_{\alpha, \beta} (\Grig_2, \Vee_{\alpha\beta})$. By Lemma~\ref{lem:PiecesSofic} each $(\Grig_2, \Vee_{\alpha\beta})$ is sofic. By Lemma~\ref{lem:SoficUnion}, $(\Grig_2, \Vee)$ is sofic.
\end{proof}

This is the first (and main) half of our main result:

\begin{lemma}
\label{lem:IsSofic}
The system $(\Grig, \Vee)$ is topologically conjugate to a sofic shift.
\end{lemma}

\begin{proof}
Set $G = \Grig, H = \Grig_2, X = X' = \Vee$. Then $H$ is a finitely-generated finite-index subgroup of $G$, $(G, X)$ is a zero-dimensional system, and $(H, X') = (\Grig_2, \Vee)$ admits an SFT cover by the previous lemma. By Lemma~\ref{lem:FiniteIndex}, the system $(\Grig, \Vee) = (G, X)$ admits an SFT cover. By Lemma~\ref{lem:IsSubshift}, it is a subshift, thus by definition it is sofic.
\end{proof}

\section{The system $(\Grig, \Vee)$ is not an SFT}
\label{sec:VNotSFT}

We now prove the second half of our main result, that $(\Grig, \Vee)$ is not an SFT.

\begin{lemma}
\label{lem:IsNotSFT}
The system $(\Grig, \Vee)$ is not conjugate to a subshift of finite type.
\end{lemma}

\begin{proof}
The idea is to show that there are shift-periodic points in arbitrarily good SFT approximations of $\Vee$ (meaning $\Z$-SFTs whose forbidden patterns are the words of length $n$ that do not appear in the language of $\Vee$), such that $\Grig$ still admits a well-defined action on their orbit. Then any set of points that locally look like these periodic points provides a pseudo-orbit that cannot be traced, since the action of $\Grig$ on $\Vee$ has no finite orbits and is expansive.

This is less trivial than it may appear. If $\Grig$ were finitely presented, it would suffice to take any good enough SFT approximation so that all the relations are ``visible'', and then the action of $\Grig$ would be well-defined on this entire SFT -- since $\Vee$ is minimal, all SFT approximations are transitive, thus periodic points are dense in them. The problem with this is that, since $\Grig$ is not finitely-presented, we do not see a general reason why such periodic orbits should exist. Below, we explain how to get them by an ad hoc argument, by modifying the words $w_n$ we used to define $\Vee$. 

Write $\wideparen{w}$ for the circular version of a word $w$, where we imagine the left and right ends are glued together. There is a natural analog of the set of SAWs $W$, $\wideparen{W}$, where the requirements are those of $W$, but we require them for all rotations of the word (equivalently, the constraint on adjacent letters is extended to the pair formed of the first and last letters), and additionally the star cannot be at the very end of the word, the point being that thinking of the word as circular, this is the same as putting the star in the beginning of the word. Similarly, we can talk about circular AWs and their starrings. We extend the jump action to $\wideparen{W}$ in the obvious way.

We claim that $\Grig$ admits a well-defined action on the starrings of $\wideparen{(w_n \alpha)^2} = \wideparen{w_n \alpha w_n \alpha}$ where $w_{n+1} = w_n \alpha w_n$. To see this, we think of the action of $H = \Z_2 * \Z_2^2$ on $\wideparen{(w_n \alpha)^2}$ as follows: Letting $I$ be the star positions of $w_n$ (including the last positions; $|I| = 2^n$), we act on $I \times \Z_2$; the action factors onto the natural action on $I$, and the action on the bit $b \in \Z_2$ is that at each end of $I$, some of the generators increase $b$ by $1$. This description makes sense because $w_n$ is a palindrome, simply imagine that the bit keeps track of which copy of $w_n$ we are in, and imagine that the rightmost copy of $w_n$ is reversed.

Next, we give an alternative description of the jump action on $w_{n+1} = w_n \alpha w_n$. Again we can think of the group as acting on $I \times \Z_2$. The action on $I$ is the exact same, but now we flip the $\Z_2$ bit if and only if we are at the right end of $w_n$.

Now suppose that a group element $g \in H$ acts trivially on the starrings of $w_n \alpha w_n$. Then in the $I \times \Z_2$ point of view, if we start at $(i, b)$ or $(2^n - 1 - i, b)$ we in particular flip the bit $b$ an even number of times. Observe that $i \mapsto 2^n - 1 - i$ is an automorphism of the action on the starrings of $w_n$, again because $w_n$ is a palindrome, so we can interpret the action on pairs $(2^n - 1 - i, b)$ as the action on $I \times \Z_2$ but now the bit $b$ is flipped when we are at the \emph{left} end of $I$. All in all, the fact we fix $(i, b)$ (resp.\ $(2^n - 1 - i, b)$) means that on $i \in I$, the action of $g$ flips $b$ on the left (resp.\ right) an even number of times. In particular, in the first action we described on $I \times \Z_2$ (corresponding to the jump action on $\wideparen{w_n \alpha w_n \alpha}$) the bit is flipped an even number of times (because $\mbox{even} + \mbox{even} = \mbox{even}$).

Note that it immediately follows that $\Grig$ also has a well-defined action on $\wideparen{w_n \alpha}$. Namely, $(H, \wideparen{w_n \alpha})$ is a quotient of $(H, \wideparen{(w_n \alpha)^2})$ under taking the star position modulo $2^n$, and if $g \in H$ acts trivially on $\wideparen{(w_n \alpha)^2}$ then in particular it acts trivially when we identify some positions.

%Next, let us extend our action of $H$ yet again, this time to a set of infinite words $X \subset \{a, B, C, D\}^\Z$. The origin plays the role of the star, and $X$ is the SFT where a subword $\alpha \beta$ is allowed iff $\alpha = a$ xor $\beta = a$; the action is again just the TFG action where now the origin jumps over letters with the usual rule (in practice we shift the word in the opposite direction).

%Now let us recall what Subshift $V$ is (or I always assumed it is, at least): a point $x \in \{a, B, C, D\}^\Z$ is there if and only if every subword of $x$ appears in one of the words $w_n$ (for some $n$). It is easy to see that $V \subset X$, so $H$ admits a well-defined action on $V$. But in fact even $\Grig$ admits a well-defined action on this subshift by the TFG action, i.e.\ if two products of generators represent the same element then they act the same way. This is obvious, since this happens on the words $w_n$ (by the assumption that they come from actual actions). The representation is also faithful: if an element of $\Grig$ is non-trivial, it acts nontrivially on some $w_n$, and the word $w_n$ appears uniformly syndetically in all the words $w_m$. Thus $(\Grig, V)$ -- which we call System $V$ -- is indeed a faithful representation of $\Grig$.

%Let $w_{n+1} = w_n \alpha w_n$.

Let $X = \{a, B, C, D\}^\Z$. Observe that although $x_n = (w_n \alpha)^\Z \in X$ is not in $\Vee$, $\Grig$ admits a well-defined action on its shift orbit. Namely, the action of $H$ on this orbit is exactly the same as that on $\wideparen{w_n \alpha}$, where we already established $\Grig$ has a well-defined action.

Observe that $x_n$ is in the $2^n$th SFT approximation of $\Vee$ seen as a $\Z$-subshift, meaning the words of length $2^n$ and less are taken from $\Vee$. This is clear, as they all appear in the word $w_{n+1}$, which in turn appears (syndetically) in all points of $\Vee$.

Because the action of $\Grig$ is well-defined on the orbit $O(x_n) = \{\sigma^k(x) \;|\; k \in \Z\}$, we can associate to $g \in \Grig$ the point $g \act_{\jump} x_n$, to obtain a map $\phi : \Grig \to O(x_n) \subset X$. Since each shift of $x_n$ is in the $2^n$th SFT approximation, we can pick for each $\phi(g) = g \act_{\jump} x_n$ an approximating point $y_g \in \Vee$ with $d(y_g, \phi(g)) \leq 2^{-n+1}$ in the Cantor metric of $X$.
 
Then because $\phi$ is an actual orbit, it follows from the triangle inequality of $X$ that $g \mapsto y_g$ is a $2^{-n+3}$-pseudo-orbit: for $\alpha$ one of the natural generators, we have
\begin{align*}
d(\alpha \act_{\jump} y_g, y_{\alpha g})
&\leq d(\alpha \act_{\jump} y_g, \alpha \act_{\jump} \phi(g)) + d(\alpha \act_{\jump} \phi(g), \phi(\alpha g)) + d(\phi(\alpha g), y_{\alpha g}) \\
&\leq 2^{-n+2} + 0 + 2^{-n+1} \\
&\leq 2^{-n+3}
\end{align*}
Where we note that the action of $\alpha$ either behaves trivially or shifts the points by $\pm 1$, which by our choice of metric on Cantor space can at most double distances.

We claim that this pseudo-orbit $(y_g)_{g \in \Grig}$ is not $1/2$-traced by any point of $\Vee$ for sufficiently large $n$. Namely suppose it were $1/2$-traced by some $z$. Then, by the definition of $1/2$-tracing and the choice of the $y_g$ we would have
\[ (g \act_{\jump} z)_0 = (y_g)_0 = (g \act_{\jump} x_n)_0. \]
The jump action of $\Grig$ on $\Vee$ is expansive, and it is easy to see that $1/2$ is an expansivity constant, i.e.\ we can shift any position of $z$ to the origin. From this we conclude that $z = x_n$. Since $x_n \notin \Vee$, this is a contradiction. So the system $(\Grig, \Vee)$ does not have the pseudo-orbit tracing property.
\end{proof}

As a ``sanity check'' we checked that the action of the group $\Grig$ on starrings of $\widehat{(w_n \alpha)^2}$ is indeed well-defined. In fact, we tested, for each $n \in [1, 6]$ and each $p \in [1, 20]$ whether the elements in $R_t$ act as identity in the $H$-action on starrings of $\widehat{(w_n \alpha)^p}$:
\[ R_t = \{a^2, b^2, c^2, d^2, bcd, \kappa^k((ad)^4), \kappa^k((adacac)^4) \;|\; 1 \leq k \leq t \} \]
Here, $\kappa$ is the substitution $\kappa(a) = aca, \kappa(b) = d, \kappa(c) = b, \kappa(d) = c$. These are so-called Lysenok relations, and $\bigcup_t R_t$ is a presentation of the group $\Grig$.

The table agrees with our result above that there is a well-defined action on starrings of $\wideparen{(w_n \alpha)^2}$. This table, and the tables where we checked the same for smaller $t$, also suggest some other things, which we did not try to prove: the action of the group $\Grig$ is well-defined on $\wideparen{(w_n \alpha)^4}$ and $\wideparen{(w_n \alpha)^8}$, but no higher powers. Furthermore, if we check only the relations $R_t$, then the first $t$ rows of the table look correct (having $1$ in slots $1, 2, 4, 8$) and the $(t+1)$th row is all $1$, i.e.\ we are not able to detect a contradiction.

\begin{table}
\begin{center}
\begin{tabular}{|l|c|c|c|c|c|c|c|c|c|c|}
\hline
$n$ $\backslash$ $p$      & 1 & 2 & 3 & 4 & 5 & 6 & 7 & 8 & 9 & 10--50 \\
\hline
$1$            & 1 & 1 & 0 & 1 & 0 & 0 & 0 & 1 & 0 & 0  \\
$2$            & 1 & 1 & 0 & 1 & 0 & 0 & 0 & 1 & 0 & 0  \\
$3$            & 1 & 1 & 0 & 1 & 0 & 0 & 0 & 1 & 0 & 0  \\
$4$            & 1 & 1 & 0 & 1 & 0 & 0 & 0 & 1 & 0 & 0  \\
$5$            & 1 & 1 & 0 & 1 & 0 & 0 & 0 & 1 & 0 & 0  \\
$6$            & 1 & 1 & 0 & 1 & 0 & 0 & 0 & 1 & 0 & 0  \\
\hline
\end{tabular}
\end{center}
\caption{Row $n$, column $p$ contains $1$ if initial Lysenok relations hold for the $H$-action on $\widehat{w_n \alpha}^p$.}
\label{tab:Test}
\end{table}

\section{The system $(\Grig, \Voro)$ is a proper sofic shift}
\label{sec:Vorobets}

We show that the system $(\Grig, \Voro)$ is also a proper sofic shift. We will deduce this entirely abstractly from our results from $(\Grig, \Vee)$.  %Soficness is immediate from observing it is a factor of $(\Grig, \Vee)$, and we show that the factor map is nice enough that we can abstractly conclude that it cannot be SFT.
%Recall that Vorobets' system is intuively defined exactly as $(\Grig, \Vee)$, but instead of thinking of $\Vee$ as a $\Z$-subshift, it is thought of as a closed system of (marked) Schreier graphs.
The system $(\Grig, \Voro)$ is obtained by quotienting the system $(\Grig, \Vee)$ by the orbit relation of the reversal map $f : \Vee \to \Vee$ defined by $f(x)_i = x_{-1-i}$. %This makes sense because $g \act_j f(x) = f(g \act_j x)$ for all $g \in \Grig, x \in \Vee$ (which in turn follows immediately from the flip-symmetric definition of the jump action).

The reversal $f$ gives a free continuous $\Z_2$-action on $X$, and thus the quotient map from $X$ to $X/f$ is exactly $2$-to-$1$, i.e.\ every point has exactly $2$ preimages. It is a general fact that a quotient map by the orbit relation of a continuous finite group action is open. Namely let $Y = X/H$ for a finite group $H$. If $U \subset X$ is open, then the full $\phi^{-1}$-preimage of its $\phi$-image is open, because $\phi^{-1}(\phi(U)) = \bigcup_{h \in H} hU$. By the definition of a quotient map, this implies $\phi(U)$ is open.

We start with two basic topological lemmas.

\begin{lemma}
\label{lem:Separated}
Let $X, Y$ be metric spaces and let $\phi : X \to Y$ be a constant-to-$1$ open continuous map. Let $X$ be compact. Then there exists $\epsilon' > 0$ such that $\phi(x) = \phi(x')$ and $x \neq x'$ imply $d(x, x') \geq \epsilon'$.
\end{lemma}

\begin{proof}
Assume $\phi$ is (exactly) $k$-to-$1$. If such $\epsilon' > 0$ did not exist, then we could find $y_i \in Y$ with preimages (exactly) $x_i^1, \ldots, x_i^k$ such that $d(x_i^1, x_i^2) < 1/i$ (possibly after reordering). Using compactness of $X^k$ we can take a diagonal limit point of the sequence of $k$-tuples $(x_i^j)_j$ as $i \rightarrow \infty$ to get preimages $(x^1, \ldots, x^k)$ for some $y$, such that $x^1 = x^2$. Note that $y = \lim_i y_i$ by continuity of $\phi$. Since there are at most $k$ distinct preimages in this tuple, we find an additional preimage $x^{k+1} \in X$. Pick a small neighborhood for $x^{k+1}$ whose closure does not contain any of the points $x^j$ with $j \leq k$. By openness of $\phi$, for any large enough $i$, $y_i$ has a preimage in $U$. For large enough $i$, these preimages differ from the preimages $(x_i^1, \ldots, x_i^k)$, so $y_i$ has at least $k+1$ preimages.
\end{proof}

The following is known in much greater generality, we give a specific proof.

\begin{lemma}
\label{lem:Continuous}
Let $X, Y$ be metric spaces and let $\phi : X \to Y$ be a constant-to-$1$ open continuous map. Then the map $y \mapsto \phi^{-1}(y)$ is continuous with respect to the Hausdorff metric.
\end{lemma}

\begin{proof}
Suppose this map were not continuous at some $y \in Y$, and let $y$ have (exactly) preimages $x^1, \ldots, x^k$. By discontinuity at this point, and the definition of the Hausdorff metric, there exists $\epsilon > 0$ such that for each $i$ we find $y_i$ such that $d(y_i, y) < 1/i$ and $y_i$ has (exactly) preimages $x_i^1, \ldots, x_i^k$, so that (possibly after renumbering) $x_1$ is not in the $\epsilon$-neighborhood of any of the points $x_i^j$. As in the previous proof, by openness, in a sufficiently small neighborhood of $x_1$ we can find preimages for $y^i$ for large enough $i$, so the map is not $k$-to-$1$.
\end{proof}

We next prove two results that show respectively that a nice enough cover preserves the property of being SFT, and a nice enough factor preserves expansivity. (Actually by similar proofs one can show that for such factor maps, the properties are preserved in both directions, i.e.\ factors preserves the SFT property and covers preserve expansivity.)

For the first result, we do not know if this result is directly in the literature, but for a related result on $\Z$, we cite \cite{BlHa91}.

\begin{lemma}
\label{lem:NoPOTP}
Let $G$ be a group generated by a finite set $S$. Let $(G,X)$ and $(G,Y)$ be group actions on compact metrizable topological spaces $X, Y$. Suppose $\phi : X \to Y$ is an open constant-to-$1$ factor map. If $Y$ has the pseudo-orbit tracing property, then so does $X$.
\end{lemma}

\begin{proof}
The idea is simply to map a pseudo-orbit $x$ from $X$ to $Y$ by using the factor map, trace the resulting pseudo-orbit $y$ in $Y$ by an actual orbit of some $y' \in Y$ using the pseudo-orbit tracing property of $Y$, and then pick a preimage $x'$ of the tracing point that is close to the preimage of $y(1)$ (using the continuity of $\phi^{-1}$ from the previous lemma). It then follows automatically from the separatedness of the $\phi^{-1}$ preimages (from Lemma~\ref{lem:Separated}) that $x'$ traces $x$.

We now prove this in detail. Let $\epsilon > 0$. We show that for some $\delta > 0$, $X$ $\epsilon$-traces $(\delta, S)$-pseudo-orbits.

Assume that the set of $\phi$-preimages of any $y \in Y$ is $\epsilon' > 0$ separated pairwise (using Lemma~\ref{lem:Separated}).  We may assume $\epsilon$ is small (since if $x \in X$ $\epsilon$-traces a pseudo-orbit, it also $\gamma$-traces it for any $\gamma > \epsilon$). In particular, we may assume $\epsilon$ is small enough that $d(x, x') < \epsilon$ implies $d(sx, sx') < \epsilon'/3$ for all $s \in S \cup \{1\}$.

Let now $\epsilon_0 > 0$ be such that if $y, y' \in Y$ and $d(y, y') < \epsilon_0$, then we have $d(\phi^{-1}(y), \phi^{-1}(y')) < \epsilon$ in Hausdorff metric (which exists by Lemma~\ref{lem:Continuous}). %Let $0 < \epsilon_1 < \epsilon_0$ be such that $d(y, y') < \epsilon_1$ implies $d(sy, sy') < \epsilon_0$ for $s \in S$.
Let $\delta_0 > 0$ be such that $Y$ $\epsilon_0$-traces $(\delta_0, S)$-pseudo-orbits. Let $\delta > 0$ be such that $d(x, x') < \delta \implies d(\phi(x), \phi(x')) < \delta_0$ for $x, x' \in X$. We may assume also $\delta < \epsilon'/3$.

Now let $x : G \to X$ be a $(\delta, S)$-pseudo-orbit. Define $y : G \to Y$ by $y(g) = \phi(x(g))$. Then $d(sy(g), y(sg)) = d(s\phi(x(g)), \phi(x(sg))) = d(\phi(sx(g)), \phi(x(sg))) < \delta_0$ for all $s \in S$, by the choice of $\delta$ and the fact $x$ is a pseudo-orbit. Then there is an actual orbit of some $y' \in Y$ that $\epsilon_0$-traces $y$, i.e.\ $d(gy', y(g)) < \epsilon_0$ for all $g \in G$. %Then also $d(sgy', sy(g)) < \epsilon_0$ for all $g \in G, s \in S \cup \{1\}$.

Since $d(y(1), y') < \epsilon_0$, we have $d(\phi^{-1}(y(1)), \phi^{-1}(y')) < \epsilon$ in Hausdorff metric, and thus we can pick $x' \in \phi^{-1}(y')$ so that $d(x(1), x') < \epsilon$. We prove by induction that for all $r$, $d(gx', x(g)) < \epsilon$ for all $g \in B_r$ (where $B_r$ is the ball of radius $r$ with respect to generators $S$). This is true for $r = 0$ by the choice of $x'$. Now suppose it is true for $g \in S^r$, and consider $s \in S$.

We have $d(sgy', y(sg)) < \epsilon_0$. The points $sgy', y(sg)$ have at least the preimages $sgx', x(sg)$ respectively. By the inductive assumption $d(gx', x(g)) < \epsilon$, from which $d(sgx', sx(g)) < \epsilon'/3$ by the choice of $\epsilon$, and then
\[ d(sgx', x(sg)) \leq d(sgx', sx(g)) + d(sx(g), x(sg)) < 2\epsilon'/3 \]
since $d(sx(g), x(sg)) < \delta < \epsilon'/3$ (since $x$ is a $(\delta, S)$-pseudo-orbit).

Since $d(\phi^{-1}(y(sg)), \phi^{-1}(sgy')) < \epsilon$ in Hausdorff metric (again by the relation of $\epsilon_0$ and $\epsilon$), there is a $\phi$-preimage $z$ for $y(sg)$ at distance less than $\epsilon$ from the point $sgx' \in \phi^{-1}(sgy')$. We claim that this preimage must be $z = x(sg)$. If not, then we have
\[ 0 < d(z, x(sg)) \leq d(z, sgx') + d(sgx', x(sg)) < \epsilon + 2\epsilon'/3 < \epsilon' \]
contradicting the assumption that $\phi^{-1}(y(sg))$ is $\epsilon'$-separated.

This shows that $d(x(sg), sgx') < \epsilon$, concluding the inductive step since $sg$ enumerates $B_{r+1}$ as we range over $g \in B_r, s \in S$. Since $\bigcup_r B^r = G$, this shows $d(gx', x(g)) < \epsilon$ for all $g \in G$, so indeed $x'$ $\epsilon$-shadows $x$, finishing the proof.
\end{proof}

\begin{lemma}
\label{lem:Expansive}
Let $G$ be a group generated by a finite set $S$. Let $(G,X)$ and $(G,Y)$ be group actions on compact metrizable topological spaces $X, Y$. Suppose $\phi : X \to Y$ is an open constant-to-$1$ factor map. If $X$ is expansive, then so is $Y$.
\end{lemma}

\begin{proof}
By Lemma~\ref{lem:Separated}, there exists $\epsilon' > 0$ such that for any $y \in Y$, $\phi^{-1}(y)$ is pairwise $\epsilon'$-separated. Let $\epsilon > 0$ be an expansivity constant for $X$. We may assume $\epsilon$ is small enough that $d(x, x') < \epsilon$ implies $d(sx, sx') < \epsilon'/2$ for all $s \in S \cup \{1\}$.

Using Lemma~\ref{lem:Continuous}, pick $\delta > 0$ such that $d(y, y') < \delta \implies d(\phi^{-1}(y), \phi^{-1}(y')) < \epsilon$. We claim that $\delta$ is an expansivity constant. Suppose that it is not. Then we can find $y, y' \in Y$ distinct, such that $d(gy, gy') < \delta$ for all $g \in G$. Then we have $d(\phi^{-1}(gy), \phi^{-1}(gy')) < \epsilon$ for all $g \in G$ by the choice of $\delta$.

Pick $x \in \phi^{-1}(y), x' \in \phi^{-1}(y')$ with $d(x, x') < \epsilon$. Analogously to the previous proof, we prove by induction on $r$ that $d(gx, gx') < \epsilon$ for all $g \in B_r$. 

If $d(gx, gx') < \epsilon$ then $d(sgx, sgx') < \epsilon'/2$. We need to show that actually $d(sgx, sgx') < \epsilon$. Since $d(sgy, sgy') < \delta$, we have $d(\phi^{-1}(sgy), \phi^{-1}(sgy')) < \epsilon$, so there is a $\phi$-preimage $z$ of $sgy'$ at distance less than $\epsilon$ from $sgx$. This must be precisely $sgx'$, since
\[ d(z, sgx') \leq d(z, sgx) + d(sgx, sgx') < \epsilon + \epsilon'/2 < \epsilon' \]
and $\phi^{-1}(sgy')$ is $\epsilon'$-separated.
\end{proof}

\begin{theorem}
The system $(\Grig, \Voro)$ is topologically conjugate to a proper sofic shift on the group $\Grig$.
\end{theorem}

\begin{proof}
Since the system $(\Grig, \Voro)$ is a factor of $(\Grig, \Vee)$, and $(\Grig,\Vee)$ is sofic, $(\Grig, \Voro)$ is also SFT covered. %The space of markeBy the previous lemma, it is a subshift (of course, this is well-known).

We now show that $(\Grig, \Voro)$ is a subshift. The space of marked Schreier graphs for any group is a compact zero-dimensional and metrizable space, so it suffices to prove expansivity. This is the content of the previous lemma. % for this. For this, the proof is analogous to the proof for the system $(\Grig, \Vee)$. (Alternatively, it is a general fact that a constant-to-$1$ open factor of a .) % is a general fact  follows immediately from the fact that the action of a group % Expansivity of the action can be proved directly, or follows from Lemma~\ref{lem:Expansive}.

Since $(\Grig, \Voro)$ is an open exactly $2$-to-$1$ factor of $(\Grig, \Vee)$, and $(\Grig,\Vee)$ is not an SFT, $(\Grig, \Voro)$ cannot be pseudo-orbit tracing by the Lemma~\ref{lem:NoPOTP}, thus it is not an SFT. In other words, $(\Grig, \Voro)$ is proper sofic.
\end{proof}

\begin{remark}
As mentioned above, lemmas~\ref{lem:NoPOTP} and~\ref{lem:Expansive} prove only one direction of an if and only if condition. By proving the other ones, we could have equivalently concentrated on the system $(\Grig,\Voro)$ in this paper, and abstractly deduced the properties of the system $(\Grig,\Vee)$.
\end{remark}

%We mentioned that $(\Grig, \Voro)$ is.

\section{The system $(\Grig, \Tree)$ is pseudo-orbit tracing}
\label{sec:TreePOTP}

The system $(\Grig, \Tree)$ is also not SFT, for the trivial reason that it is not expansive, so not a subshift in the first place. As we explained in Section~\ref{sec:Preliminaries}, among subshifts SFTs are characterized by the pseudo-orbit tracing property. Perhaps surprisingly (in this light), for very general reasons, $(\Grig, \Tree)$ does have the pseudo-orbit tracing property.

\begin{lemma}
Let $X$ be a compact metrizable zero-dimensional space, let $G$ be a finitely-generated group acting on $X$ equicontinuously. Then $(G, X)$ has the pseudo-orbit tracing property.
\end{lemma}

\begin{proof}
We may assume $X \subset \{0,1\}^\N$ is a closed set. Let $m \in \N$. In terms of coordinates, equicontinuity means that there exists $n \in \N$ such that for $x \in X$, the word $x|_{[0, n-1]}$ determines the $G$-tuple $T(x) = ((g \act x)|_{[0, m-1]})_{g \in G}$ uniquely. A direct calculation shows that these $G$-tuples form a $G$-subshift $Y$ over alphabet $\{0, 1\}^m$, and that the map from $x$ to $T(x)$ is a surjective factor map from $X$ to $Y$ as $G$-systems. The subshift $Y$ is finite, as we showed it has at most $\{0,1\}^n$ points.

A finite factor for $(G, X)$ is equivalent to a finite partition of $X$ into clopen sets $C_1, \ldots, C_k$, such that the action of $G$ respects the partition in the sense that for all $i$, there exists $j$ such that $g \act C_i = C_j$. Furthermore, from the construction is it clear that the diameters of the $C_i$ can be made arbitrarily small -- the diameters of the $C_i$ are at most those of the cylinders $[x|_{[0,m-1]}]$.

Let $\epsilon > 0$ be arbitrary and construct an invariant clopen partition $(C_i)_{i=1}^k$ with each $C_i$ having diameter at most $\epsilon$. Since the sets $C_i$ are compact, we can find $\delta > 0$ with $\delta < \epsilon$ such that if $d(x, y) < \delta$ and $x \in C_i$, then $y \in C_i$ as well. For $x \in X$, write $C(x) = i$ for the unique $i$ such that $x \in C_i$.

Let now $F$ be any finite generating set for $G$ and let $z : G \to X$ be a $(\delta, F)$-pseudo-orbit. Let $x = z(1_G)$. We prove by induction on $r$ that for $g \in B_r$ (the ball of radius $r$ in $G$ with generators $F$) we have $C(z(g)) = C(gx)$, from which we then conclude $d(z(g), gx) < \epsilon$ for all $g \in G$, proving the pseudo-orbit tracing property.

For this, suppose the claim is true for $g \in B_r$ and consider any $s \in F$. Then by the assumption that the partition $(C_i)_i$ is respected by $G$, we have that $C(z(g)) = C(gx)$ implies $C(sz(g)) = C(sgx)$. By the assumption that $z$ is a pseudo-orbit, we have $d(sz(g), z(sg)) < \delta$, so by the choice of $\delta$ also $C(z(sg)) = C(sz(g)) = C(sgx)$, which concludes the proof since $sg$ gives all elements of $B_{r+1}$.
\end{proof}

\begin{proposition}
The system $(\Grig, \Tree)$ has the pseudo-orbit tracing property.
\end{proposition}

\begin{proof}
This system satisfies the assumptions of the previous lemma.
\end{proof}

\bibliographystyle{plain}
\bibliography{bib}{}

\end{document}